\documentclass[letterpaper]{amsart}

\usepackage[utf8]{inputenc}
\usepackage{amsmath}
\usepackage{amssymb}
\usepackage{amsthm}
\usepackage{tikz-cd}
\usepackage{mathrsfs}
\usepackage{enumitem}
\usepackage[alphabetic]{amsrefs}
\usepackage{tikz}
\usepackage{float}
\usepackage{hyperref}

\theoremstyle{plain}
\newtheorem{teo}{Theorem}[section]
\newtheorem{prop}[teo]{Proposition}
\newtheorem{lemma}[teo]{Lemma}
\newtheorem{cor}[teo]{Corollary}
\theoremstyle{definition}

\newtheorem{condition}[teo]{Condition}

\newtheorem{question}[teo]{Question}
\theoremstyle{remark}
\newtheorem{remark}[teo]{Remark}
\newtheorem{claim}[teo]{Claim}
\numberwithin{equation}{section}

\newcommand{\A}{\mathcal{A}}
\newcommand{\B}{\mathcal{B}}
\newcommand{\C}{\mathbb{C}}

\newcommand{\Cheart}{\mathcal{C}}
\newcommand{\D}{\mathcal{D}}
\newcommand{\DC}{\mathrm{D}}

\newcommand{\Ftors}{\mathcal{F}}
\newcommand{\HC}{\mathcal{H}}

\renewcommand{\O}{\mathscr{O}}
\renewcommand{\P}{\mathbb{P}}
\newcommand{\Pslicing}{\mathcal{P}}
\newcommand{\Q}{\mathbb{Q}}
\newcommand{\R}{\mathbb{R}}
\newcommand{\Ttors}{\mathcal{T}}

\newcommand{\Z}{\mathbb{Z}}

\let\Re\relax
\let\Im\relax
\DeclareMathOperator{\Amp}{Amp}
\DeclareMathOperator{\ch}{ch}
\DeclareMathOperator{\Coh}{Coh}
\DeclareMathOperator{\Exc}{Exc}
\DeclareMathOperator{\Ext}{Ext}
\DeclareMathOperator{\Hom}{Hom}
\DeclareMathOperator{\Im}{Im}
\DeclareMathOperator{\Nef}{Nef}
\DeclareMathOperator{\NS}{NS}
\DeclareMathOperator{\Re}{Re}

\DeclareMathOperator{\Stab}{Stab}
\DeclareMathOperator{\Supp}{Supp}

\newcommand{\EXT}{\mathcal{E}\kern -.5pt xt}
\newcommand{\HOM}{\mathcal{H}\kern -.5pt om}

\newcommand{\abs}[1]{\left\lvert#1\right\rvert}
\newcommand{\norm}[1]{\left\lVert#1\right\rVert}

\begin{document}

\title{Stability conditions on surfaces and contractions of curves}
\author{Nicolás Vilches}
\address{Department of Mathematics, Columbia University, 2990 Broadway, New York, NY 10027, USA}
\email{nivilches@math.columbia.edu}
\begin{abstract}
We study Bridgeland stability conditions on smooth surfaces arising from birational morphisms $S \to T$ to a singular surface. Assuming that $T$ has only ADE singularities or certain cyclic quotient singularities, we produce pre-stability conditions on $S$ whose central charges depend on the pullback of an ample divisor on $T$. Moreover, we prove the support property for the cyclic quotient case (including the $A_n$ singularities). These stability conditions arise as limits of the Arcara--Bertram stability conditions on $S$.

In a complementary direction, we study birational maps $S \to T$ for which a stability condition $S$ can be obtained using the pullback of an ample class on $T$. We prove that the morphism cannot contract smooth curves of positive genus. 
\end{abstract}

\maketitle

\tableofcontents

\section{Introduction}

Let $S$ be a smooth, projective surface over the complex numbers. The existence of stability conditions on $\DC^b(S)$ was settled by \cite{Bri08} for K3 surfaces, and \cite{AB13} in full generality. More precisely, we have the following theorem.

\begin{teo}[Arcara--Bertram, \cite{AB13}] \label{teo:intro_AB}
Let $S$ be a smooth, projective surface, and let $\beta \in \NS(S)_\R$, $\omega \in \Amp(S)_\R$ be two numerical classes. Then there exists a Bridgeland stability condition $\sigma_{\beta, \omega} = (\A_{\beta, \omega}, Z_{\beta, \omega})$ on $\DC^b(S)$ whose central charge $Z_{\beta, \omega}$ is given by the expression
\begin{equation} \label{eq:intro_ABcentral}
Z_{\beta, \omega}(E) = -\ch_2^\beta(E) +\frac{\omega^2}{2}\ch_0(E) + i\omega.\ch_1^\beta(E).
\end{equation}
\end{teo}

Crucially, all these stability conditions fit lie in the so-called \emph{geometric chamber} of $\Stab(S)$: the open subset of $\Stab(S)$ where skyscraper sheaves are stable objects of the same phase. Understanding the boundary of this chamber is not an easy task, pioneered by \cite{Bri08} to describe the stability manifold for K3 surfaces. This motivates the following leading question.

\begin{question} \label{question:intro_main}
Describe the closure in $\Stab(S)$ of the set of stability conditions
\[ \{\sigma_{\beta, \omega} : \beta \in \NS(S)_\R, \, \omega \in \Amp(S)_\R\} \]
determined by the Arcara--Bertram construction.
\end{question}

By continuity, such a stability condition will have a central charge of the form $Z_{\beta, \lambda}$ as in \eqref{eq:intro_ABcentral}, where $\beta \in \NS(S)_\R$ and $\lambda \in \Nef(S)_\R$. We will denote these stability conditions by $\overline{\sigma}_{\beta, \lambda}$, to keep in mind that these are obtained by a limiting process (and not by Theorem \ref{teo:intro_AB}).

This question has been already addressed for some big and nef divisors $\lambda=f^\ast \eta$, where $f\colon S \to T$ is a birational map to a projective, normal surface $T$. In \cite{Tod13}*{Theorem 1.2} it was shown the existence of a stability condition $\overline{\sigma}_{0, f^\ast \eta}$, where $f$ is a blow-down of a $(-1)$-curve (and so $T$ is smooth). See also \cite{Tod14}*{Theorem 1.3}, where a related result is proven when $f$ is a composition of blow-ups.

Later on, \cite{TX22}*{Theorem 5.4} showed their existence by replacing $f$ with the contraction of a $(-n)$-curve. More precisely, they produced stability conditions $\overline{\sigma}_{\beta, f^\ast \eta}$, where $f\colon S \to T$ contracts a smooth, rational curve of self-intersection $-n$, and $\beta.C+n/2$ is not an integer. It is worth noting that this restriction on $\beta$ is related to the sheaves $\O_C(k)$ on $C$.

When $S$ is a K3 surface containing $(-2)$-curves, the results of \cite{Bri08} show that stability conditions of the form $\overline{\sigma}_{\beta, f^\ast\eta}$ exist, where $f\colon S \to T$ is crepant. Bridgeland's result also includes a restriction on $\beta$, related to the fact that if $C \subset S$ is a $(-2)$-curve, then the sheaves $\O_C(k)$ are spherical objects. 

Lastly, Question \ref{question:intro_main} has also been addressed when $f\colon S \to T$ is the minimal resolution of a single ADE singularity by the very recent \cite{Cho24}. In loc. cit., it is also shown that the stability conditions $\overline{\sigma}_{\beta, f^\ast\eta}$ are directly related to stability conditions on the \emph{singular} surface $T$. We point out that our work was done independently of Chou's article.

Our main result produces many more examples of stability conditions in the direction of Question \ref{question:intro_main}. On one hand, we will allow the contraction of various chains of rational curves. On the other hand, we will allow contractions of disjoint chains simultaneously.

\begin{teo} \label{teo:intro_existence}
Let $f\colon S \to T$ be a birational morphism from a smooth, projective surface $S$ to a normal surface $T$. Assume that each connected component $C$ of $\Exc(f)$ falls into one of the following types:
\begin{itemize}
\item (ADE) $C$ is the exceptional divisor of the minimal resolution of an ADE singularity.

\item (Chain) $C=C_1 \cup \dots \cup C_r$ is a chain of rational curves, with $C_i^2 + k <0$, where $k$ is the number of curves intersecting $C_i$.
\end{itemize}

Let $\beta \in \NS(S)_\Q$ be general. Then, for any $\eta \in \Amp(T)_\Q$, there is a pre-stability condition $\overline{\sigma}_{\beta, f^\ast\eta}$ on $S$, whose central charge is $Z_{\beta, f^\ast\eta}$. 

Moreover, assume that each component of $\Exc(f)$ is a chain. Then there is an open subset of $\NS(S)_\R$ such that for $\beta$ in this open subset $\overline{\sigma}_{\beta, f^\ast\eta}$ satisfies the support property. We also have that the $\overline{\sigma}_{\beta, f^\ast\eta}$ produced in this fashion are the limit (in the topology of $\Stab(S)$) of the Arcara--Bertram stability conditions $\sigma_{\beta, \omega}$ as $\omega$ approaches $f^\ast \eta$.
\end{teo}

Let us point out that the generality condition on $\beta$ can be made explicit, although in a cumbersome way. For instance, if $\Exc(f)$ contracts only one chain of rational curves $\{C_1, \dots, C_r\}$, the condition can be stated as
\[ \sum_{i=a}^b \left( \beta.C_i + \frac{C_i^2}{2} \right) \notin \Z, \qquad \forall 1 \leq a \leq b \leq r, \]
We will precisely describe the restriction on $\beta$ in Condition \ref{cond:prestab_beta}. These restrictions are related the fact that given any chain of curves $C_1 \cup \dots \cup C_r$ and integers $d_1, \dots, d_r$, we have a line bundle $\O(d_1, \dots, d_r)$ supported on this chain, cf. \cite{DG01}*{2.7}. The assumption on $\beta$ ensures that $\ch_2^\beta(\O(d_1, \dots, d_r))$ is never zero.

Let us briefly compare our strategy to the previous works \citelist{\cite{Tod13} \cite{TX22}}. In their case, the most delicate step is to construct an appropriate heart for the stability conditions. This is obtained by a double-tilt argument: first tilt as in \cite{AB13}, and then tilt an appropriate collection of objects supported on the exceptional locus on $f\colon S \to T$. 

Our approach instead is based on the following observation (cf. Lemma \ref{lemma:heart_indep}): the real part of $Z_{\beta, \omega}$ depends only on $\beta$ and $\omega^2$, and \emph{not} on $\omega$ by itself. This will allow us to construct a skewed\footnote{See the discussion in Remark \ref{remark:basicbrid_heart}.} heart for our stability conditions without effort. Using this heart, we prove in Proposition \ref{prop:cc_crit} a criterion to verify whether $Z_{\beta, f^\ast\eta}$ defines a central charge on this heart. The condition roughly translates to proving that $Z_{\beta, f^\ast\eta}(E)\neq 0$, if $E$ is in the heart.

The next step is to describe the possible objects $E$ in the heart satisfying that $Z_{\beta, f^\ast\eta}(E)=0$. Assuming only that $f\colon S \to T$ is birational, we can assume that $E$ is $\sigma_{\beta, \omega}$-stable (with respect to an auxiliary ample $\omega$), whose support is contained in a connected component of $\Exc(f)$. By dealing with the ADE and chain cases separately, we conclude the existence of the pre-stability conditions. 

At last, the support property requires a careful work with inequalities. The strategy mostly follow the ideas of \cite{Tod13}, adapting the arguments to deal with chains of rational curves.

As we mentioned above, one of the key differences of our approach is that we identify the heart of the stability conditions $\overline{\sigma}_{\beta, f^\ast\eta}$ directly from the hearts of $\sigma_{\beta, \omega}$. This allows us to address Question \ref{question:intro_main} in a negative direction, giving certain restrictions for $\lambda \in \Nef(S)_\R$, as the following result shows.

\begin{teo} \label{teo:intro_nonexistence}
Let $S$ be a smooth, projective surface. Assume that there is a birational morphism $f\colon S \to T$ to a normal, projective surface, an ample class $\eta \in \Amp(T)_\R$, and a numerical class $\beta \in \NS(T)_\R$ such that a stability condition $\overline{\sigma}_{\beta, f^\ast\eta}$ exists as a limit (in the topology of $\Stab(S)$) of the Arcara--Bertram stability conditions $\sigma_{\beta', \omega}$. Then, no smooth curve of positive genus can be contained in the exceptional locus of $f$.
\end{teo}

In subsequent work we will use the results of Theorem \ref{teo:intro_existence} to describe various moduli spaces of Bridgeland semistable objects with numerical class $[pt]$. This will extend the discussion of \citelist{\cite{Tod13}*{\textsection 3.5} \cite{TX22}*{\textsection 7}} to produce more complicated moduli spaces. This will also include a description of their local structure using differential graded Lie algebras.

\subsection{Structure of the paper}

We will start by reviewing the basics of Bridgeland stability conditions in Section \ref{sec:stabcond}. Then, we will discuss the heart for the stability conditions in Section \ref{sec:constrheart}. This will give us immediately the proof of Theorem \ref{teo:intro_nonexistence} in Subsection \ref{subsec:pfnonexistence}.

The next two sections will be devoted to the proof of Theorem \ref{teo:intro_existence}. We will first prove the existence of pre-stability conditions in Section \ref{sec:prestab}. After that, we will show the support property in Section \ref{sec:supp}.

\subsection{Conventions}

We will work over the complex numbers. Given a normal projective surface $T$, and two curves $C, D$ in $T$, we let $C.D \in \Q$ to be Mumford's intersection product (\cite{Mum61}*{p. 17}). If $C$ is $\Q$-Cartier, this agrees with $\frac{1}{m} \deg_D \O_S(mC)|_C$, where $m$ is the Cartier index of $C$. In particular, if $T$ has $\Q$-factorial singularities (e.g. $T$ is smooth, or if $T$ only has quotient singularities), this agrees with the usual intersection product. We extended this to all Weil divisors by linearity.

We let $\NS(T)_\Q$ to be the vector space of $\Q$-Weil divisors on $T$, modulo numerical equivalence (with respect to the intersection product above), and similarly for $\NS(T)_\R$. We refer to \cite{Mum61}*{pp. 17--18} for the basic properties.

In particular, if $f\colon S \to T$ is a birational map, then we have a pullback map $f^\ast\colon \NS(T)_\Q \to \NS(S)_\Q$. This preserves the intersection product. The orthogonal complement of the image of $f^\ast$ agrees with the kernel of $f_\ast$; a basis is given by the exceptional curves of $f$ (\cite{Mum61}*{p. 6}). 

Given a smooth, projective surface $S$ and $E \in \DC^b(S)$, we let $\ch(E)$ to be the Chern character of $E$. Given $\beta \in \NS(S)_\R$, we set $\ch^\beta(E) = \ch(E).\exp(-\beta)$. More explicitly, we have 
\begin{align*}
\ch_0^\beta(E) &= \ch_0(E), \\ \ch_1^\beta(E) &= \ch_1(E)-\beta.\ch_0(E), \\ \ch_2^\beta(E) &= \ch_2(E)-\beta.\ch_1(E) + \frac{\beta^2}{2}\ch_0(E).
\end{align*}

If $\A \subset \DC^b(S)$ is the heart of a bounded t-structure on $\DC^b(S)$, we denote by $\HC^i_\A(E) \in \A$ the cohomology objects of an element of $E \in \DC^b(S)$ with respect to $\A$ (or simply $\HC(E)$ is the heart is understood). For the standard heart $\A=\Coh(S)$, we will write $H^i(E)$ instead. 

If $A, B \in \DC^b(S)$, we will denote $\Ext^i(A, B) = \Hom(A, B[i])$. We have composition maps $\Ext^i(B, C) \times \Ext^j(A, B) \to \Ext^{i+j}(A, C)$, which we denote by $\circ$.

\subsection{Acknowledgments}

First and foremost, I would like to thank my PhD advisor, Giulia Saccà. Her expertise and constant support have been a fundamental pillar of this project. 

I would like to thank James Hotchkiss, Johan de Jong, Emanuele Macrì, and Giancarlo Urzúa for various discussions related to this work. Especially, I express my gratitude to Arend Bayer and Tzu-Yang Chou, who informed me of the work \cite{Cho24}.

At last, I am grateful to the Mathematisches Forschungsinstitut Oberwolfach and the organizers of the workshop ``Algebraic Geometry: Wall Crossing and Mo\-du\-li Spaces, Varieties and Derived Categories'', where a preliminary version of this project was presented in July 2024.

\section{Preliminaries on Bridgeland stability conditions} \label{sec:stabcond}

\subsection{Basic definitions} \label{subsec:basicbrid}

We follow \citelist{\cite{Bri07} \cite{Bay19}}. Consider a triangulated category $\D$ together with an homomorphism $v\colon K(\D) \to \Lambda$ to a finite dimensional lattice $\Lambda$. A \emph{Bridgeland stability condition} on $(\D, \Lambda, v)$ (or on $\D$, if $\Lambda, v$ are understood) is given by a pair $\sigma=(Z, \Pslicing)$ described as follows.
\begin{enumerate}[label=(\alph*)]
\item $\Pslicing$ is a \emph{slicing} of $\D$: a collection of full subcategories $\{ \Pslicing(\phi) \}_{\phi \in \R}$ subject to the following relations:
\begin{itemize}
\item For all $\phi$, we have $\Pslicing(\phi+1) = \Pslicing(\phi)[1]$.
\item Given $\phi_1>\phi_2$ and $E_i \in \Pslicing(\phi_i)$, we have $\Hom(E_1, E_2)=0$. 
\item For any $E \in \D$, there is a sequence of maps 
\[ 0=E_0 \xrightarrow{i_0} E_1 \xrightarrow{i_1} \dots \xrightarrow{i_m} E_m=E, \]
and real numbers $\phi_1>\dots > \phi_m$, such that the cone of $i_k$ is in $\Pslicing(\phi_k)$ for each $k$. The cones are called the \emph{Harder--Narasimhan factors} of $E$.
\end{itemize}

\item $Z\colon \Lambda \to \C$ is a $\Z$-linear map, called the \emph{central charge}. 
\end{enumerate}
We write $Z(E)=Z(v(E))$. We impose the following compatibility condition:
\begin{enumerate}[label=(\alph*), resume]
\item For each non-zero $E \in \Pslicing(\phi)$, we have $Z(E) \in \R_{>0}\cdot \exp(i\pi \phi)$.
\end{enumerate}
Properties (a)--(c) define a \emph{pre-stability condition}. The objects of $\D$ that lie in some $\Pslicing(\phi)$ are called \emph{semistable} of \emph{phase} $\phi$. 

At last, we impose the so-called \emph{support property}, for which two equivalent formulations are possible.
\begin{enumerate}[label=(\alph*), resume]
\item There exists a quadratic form $Q$ on $\Lambda_\R$, such that $Q|_{\ker v}$ is negative definite, and $Q(E)\geq 0$ for every $E$ semistable.
\item[(d')] We have that $\abs{Z(E)} \geq \alpha\lVert v(E) \rVert$ for all $E$ semistable, where $\lVert \cdot \rVert$ is some norm in $\Lambda_\R$ and $\alpha>0$.
\end{enumerate}

\begin{remark}[cf. \cite{Bri07}*{5.3}] \label{remark:basicbrid_heart}
Given a pre-stability condition $(Z, \Pslicing)$, the slicing $\Pslicing$ defines the heart of a bounded t-structure $\A=\Pslicing((0, 1])$. This construction can be reversed, and so we can define a pre-stability condition as a pair $(Z, \A)$, where $\A$ is the heart of a bounded t-structure, $Z$ maps $\A$ to the upper half-space, and such that objects in $\A$ admit an appropriate Harder--Narasimhan filtration.

Alternatively, let us fix some value $b \in \R$. We could consider the subcategory $\B=\Pslicing((b, b+1])$ instead, which is also the heart of a bounded t-structure on $\D$. This way, we say that $Z$ is a \emph{skewed} central charge in $\B$, following \cite{Bri07}*{4.4}.
\end{remark}

\subsection{Arcara--Bertram construction} \label{subsec:AB}

Let us briefly recall some details of the construction of the stability conditions $\sigma_{\beta, \omega}$ from Theorem \ref{teo:intro_AB}. We will focus on the situation when $\beta, \omega$ have rational coefficients.

We start by constructing a heart $\A_{\beta, \omega}$ on $\DC^b(S)$, by tilting $\Coh(S)$ with respect to the torsion pair 
\begin{align*}
\Ttors_{\beta, \omega} &= \langle \{ E: E \text{ torsion} \} \cup \{E: E \text{ torsion-free, stable, } \mu(E)>\beta.\omega \} \rangle, \\
\Ftors_{\beta, \omega} &= \langle \{ E: E \text{ torsion-free, stable, }\mu(E)\leq \beta.\omega\} \rangle.
\end{align*}
Here $\langle - \rangle$ denotes the smallest subcategory of $\Coh(S)$ closed under extensions. 

As mentioned in \eqref{eq:intro_ABcentral}, the central charge $Z_{\beta, \omega}$ is given by the expression
\[ Z_{\beta, \omega}(E) = -\ch_2^\beta(E) +\frac{\omega^2}{2}\ch_0(E) + i\omega.\ch_1^\beta(E), \]
and $\Lambda, v$ is given by the Chern character. Here \cite{AB13} shows that $(Z_{\beta, \omega}, \A_{\beta, \omega})=:\sigma_{\beta, \omega}$ is a stability condition on $S$. Moreover, the assignment $(\beta, \omega) \mapsto \sigma_{\beta, \omega}$ extends to a \emph{continuous} map $\Sigma\colon \NS(S)_\R \times \Amp(S)_\R \to \Stab(S)$. We denote by $\Pslicing_{\beta, \omega}$ the slicing of $\DC^b(S)$ induced by the stability condition $\sigma_{\beta, \omega}$.

Another important fact that we will require later has to do with the support property for the Arcara--Bertram construction. This is an involved construction, so we mention some of the key steps. First of all, it is easy to verify the following ``weak'' support property.

\begin{lemma}[\cite{MS17}*{6.13}] \label{lemma:AB_presupport}
Let $E \in \DC^b(S)$ be a $\sigma_{\beta, \omega}$-semistable object. Then
\[ (\ch_1^\beta(E).\omega)^2 \geq 2\omega^2\ch_0(E)\ch_2^\beta(E). \]
\end{lemma}

As mentioned in \cite{MS17}*{\textsection 6.3}, this inequality allows us to obtain a well-behaved wall and chamber structure for $\{ \sigma_{\beta, s\omega} \}_{s>0}$. This is used to prove the support property in full generality.

\begin{prop}[\cite{MS17}*{6.13}]
Let $\omega \in \Amp(S)$ be ample, and fix a constant $C_\omega>0$ such that
\begin{equation} \label{eq:AB_valueomega}
\frac{C_\omega}{\omega^2} (\omega.D)^2 + D^2 \geq 0
\end{equation}
for any effective divisor $D \subset S$. Then the stability condition $\sigma_{\beta, \omega}$ satisfies the support property with respect to the quadratic form
\begin{equation} \label{eq:AB_fullsupport}
Q(E) = \ch_1^\beta(E)^2 -2\ch_0(E)\ch_2(E) + \frac{C_\omega}{\omega^2} (\ch_1^\beta(E).\omega)^2.
\end{equation}
\end{prop}

Note that the constant $C_\omega$ satisfies \eqref{eq:AB_valueomega} if we replace $\omega$ by $s\omega$. This way, one easily verifies that \eqref{eq:AB_valueomega} depends only on $\beta$ and the ray spanned by $\omega$.

\section{Construction of the heart} \label{sec:constrheart}

Let us fix a smooth, projective surface $S$. Recall from Question \ref{question:intro_main} that we wish to describe stability conditions $\overline{\sigma}_{\beta, \lambda}$ that are limits of the Arcara--Bertram construction, where $\lambda$ is a nef class. We will focus on the case when $\lambda$ is big and nef, so that $\lambda^2>0$. We will give a description of a twisted heart for such stability conditions, and a criterion to check whether $Z_{\beta, \lambda}$ defines a pre-stability condition on it. At last, \emph{assuming} that $\overline{\sigma}_{\beta, \lambda}$ exists, we will give a description of the heart. We will use this description to prove Theorem \ref{teo:intro_nonexistence}.

\subsection{Shared heart}

Fix $\beta \in \NS(S)_\R$ and a positive number $V>0$. Consider the collection $\Amp_V = \{ \omega \in \Amp(S)_\R : \omega^2 = V \}$. We point out that the intersection form on $\Amp(S)$ is defined over $\Q$. This way, if the collection of $\lambda \in \NS(S)_\R$ with $\lambda^2=V$ contains a single class with rational coefficients, then it has a dense set of classes with rational coefficients. 

For each ample class $\omega \in \Amp_V$, we consider the stability condition $\sigma_{\beta, \omega}$ from Subsection \ref{subsec:AB}. As we mentioned there, the heart of these stability conditions are obtained by tilting $\Coh(S)$ with respect to $\omega$-slope. Thus, in general these hearts will be different. However, the following simple lemma shows that there \emph{is} a shared heart, obtained by tilting the standard heart.

\begin{lemma} \label{lemma:heart_indep}
The heart $\B_{\beta, V} := \Pslicing_{\beta, \omega}((-1/2, 1/2])$ does not depend on $\omega \in \Amp_V$. Similarly, the subcategory $\Cheart_{\beta, V} = \Pslicing_{\beta, \omega}(1/2)$ depends only on $\beta$ and $\omega^2$
\end{lemma}

\begin{proof}
We have that
\[ \Re Z_{\beta, \omega}(E) = -\ch_2^\beta(E) + \frac{\omega^2}{2} \ch_0(E) = -\ch_2^\beta(E) + \frac{V}{2} \ch_0(E). \]
This way, the central charges $Z_{\beta, \omega}$ depend only on $\beta$ and $V$. From here the (skewed) heart $\B_{\beta, V}$ does not depend on $\omega$, cf. \cite{Bay19}*{\textsection 5}. The second result is proved similarly.
\end{proof}

As a consequence, note that $(Z_{\beta, \omega}, \B_{\beta, V})$ is a skewed stability condition. We point out that $\Cheart_{\beta, V}$ is of finite length, cf. \cite{Bri08}*{4.5}. In this category, the minimal objects (i.e. without subobjects) correspond to the $\sigma_{\beta, \omega}$-stable objects of phase $1/2$ for some (hence for all) $\omega \in \Amp_V$.

\subsection{Criterion for being a central charge} \label{subsec:cc}

Fix $\beta \in \NS(S)_\Q$ and $V>0$. Note that we are assuming that $\beta$ has rational coefficients.

\begin{prop} \label{prop:cc_crit}
Let $\lambda \in \Nef(S)_\Q$ be a class with $\lambda^2=V$. We have that $Z_{\beta, \lambda}$ defines an skewed central charge on $\B_{\beta, V}$ if and only if there are no minimal objects $E$ in $\Cheart_{\beta, V}$ with $\Im Z_{\beta, \lambda}(E)=0$.
\end{prop}

\begin{proof}
Pick a sequence $\omega_i \in \Amp(S)_\Q$ with $\omega_i^2 = V$ and $\omega_i$ approaching $\lambda$. The fact that $Z_{\beta, \omega_i}$ defines a skewed central charge on $\B_{\beta, V}$ gives us $E \in \B_{\beta, \omega_i}$, hence $\Re Z_{\beta, \omega_i}(E)\geq 0$. Moreover, if equality holds and $E$ is non-zero, then $\Im Z_{\beta, \omega_i}(E)>0$. By continuity, we get that for each $E \in \B_{\beta, V}$ we have $\Re Z_{\beta, \lambda}(E)\geq 0$, and in case of equality $\Im Z_{\beta, \lambda}(E)\geq 0$.

This way, any object $E \in \B_{\beta, V}$ with $\Re Z_{\beta, \lambda}(E)=0$ must be in $\Cheart_{\beta, V}$ by definition. Here any minimal object has $\Im Z_{\beta, \lambda}(E) \geq 0$. Thus, if there are no objects with $\Im Z_{\beta, \lambda}(E)=0$, then $Z_{\beta, \lambda}$ maps the objects on $\B_{\beta, V}$ into the upper half-space. The Harder--Narasimhan property is a direct consequence of $\beta, \lambda$ being rational and $\Cheart_{\beta, V}$ being of finite length, thanks to \cite{MS17}*{4.10}.

This proves one direction of the proposition; the other one is clear.
\end{proof}

Verifying that this condition holds for a particular combination of $\beta$ and $\lambda$ is not easy, as it requires some explicit control on the collection $\Cheart_{\beta, V}$. Nevertheless, some basic description can be obtained.

\begin{lemma} \label{lemma:cc_descriptionbad}
Fix $\beta \in \NS(S)_\Q$ and $V>0$. Let $f\colon S \to T$ be a birational map to a normal, projective surface $T$. Pick $\eta \in \Amp(T)_\Q$ a class with $\eta^2=V$, and let $\lambda=f^\ast\eta \in \Nef(S)_\Q$.

Let $0 \neq E \in \Cheart_{\beta, V}$ be a minimal object, and suppose that $\Im Z_{\beta, \lambda}(E)=0$. Then $E$ is a coherent sheaf on $S$, pure of dimension 1, whose support is connected and contained in $\Exc(f)$. 
\end{lemma} 

\begin{proof}
Let us start by fixing a collection of ample, rational classes $\omega_i$ approaching $\lambda$, satisfying $\omega_i^2=V$. Note that the fact that $E \in \Cheart_{\beta, V}$ implies
\begin{equation} \label{eq:cc_realpart}
0 = \Re Z_{\beta, \omega_i}(E) = -\ch_2^\beta(E) + \frac{V}{2} \ch_0(E).
\end{equation}
We also have $\Im Z_{\beta, \omega_i}(E) = \omega_i.\ch_1^\beta(E) >0$, as the $Z_{\beta, \omega_i}$ do define central charges. 

Let us start by showing that $E$ has $\ch_0(E)=0$. In fact, the stability of $E$ implies
\[ (\omega_i.\ch_1^\beta(E))^2 \geq 2 \omega_i^2 \ch_0(E) \ch_2^\beta(E) = V^2 \ch_0(E)^2, \]
by Lemma \ref{lemma:AB_presupport} and by \eqref{eq:cc_realpart}. Taking the limit as $i \to \infty$ and using that $\Im Z_{\beta, \lambda}(E)=0$ yields that $\ch_0(E)=0$.

We claim that $\ch_0(H^0(E))=0$. Otherwise, we let $F_i$ to be the last Harder--Narasimhan factor of $H^0(E)$ with respect to $\omega_i$-slope. This gives us the short exact sequence
\begin{equation} \label{eq:cc_lastfactor}
0 \to G_i \to E \to F_i \to 0
\end{equation}
in $\A_{\beta, \omega_i}$. (Note that $F_i$ will most certainly depend on $i$.) As $E$ is $\sigma_{\beta, \omega_i}$-stable with $\Re Z_{\beta, \omega_i}(E)=0$, this implies that $\ch_2^\beta(F_i)> V\ch_0(F_i)/2$. But $F_i$ is $\omega_i$-slope semistable, hence 
\[ \ch_2^\beta(F_i) \leq \frac{\ch_1^\beta(F_i)^2}{2\ch_0(F_i)} \leq \frac{(\ch_1^\beta(F_i) \cdot \omega_i)^2}{2 V \ch_0(F_i)} \]
by the Bogomolov--Gieseker inequality and the Hodge Index Theorem. Putting these two inequalities together yields $(\ch_1^\beta(F_i) \cdot \omega_i)^2 \geq V^2$.
At last, from \eqref{eq:cc_lastfactor} we get $\Im Z_{\beta, \omega_i}(E) \geq \Im Z_{\beta, \omega_i}(F_i)\geq 0$, and so $(\ch_1^\beta(E) \cdot \omega_i)^2 \geq V^2$. Taking the limit as $\omega_i \to \lambda$ yields a contradiction. It follows that $\ch_0(H^0(E))=0$.

Note that $\ch_0(H^0(E))=0$ and $\ch_0(E)=0$ together imply that $\ch_0(H^{-1}(E))=0$. This way, it follows that $H^{-1}(E)=0$, and that $E=H^0(E)$ is a torsion sheaf. The fact that $\Im Z_{\beta, \lambda}(E)=0$ translates to $\eta.f_\ast \ch_1(E)=0$. But $\ch_1(E)$ is effective, hence the only option is that $E$ is supported on $\Exc(f)$ union finitely many points. Here $E$ must have $\ch_1(E)\neq 0$, as otherwise $\Im Z_{\beta, \omega_i}(E)$ will be zero.

Note that $E$ has no subsheaves of dimension zero. In fact, any such subsheaf will destabilize $E$ in $\sigma_{\beta, \omega_i}$, as all these have phase 1. Similarly, if the support of $E$ is disconnected, then $E$ could we written as a direct sum, contradicting the fact that $E$ is stable. This proves the required result.
\end{proof}

\subsection{Identifying the heart} \label{subsec:idheart}

Fix a birational map $f\colon S \to T$, and consider $\beta \in \NS(S)_\Q$, $\eta \in \Amp(S)_\Q$. Let $\lambda=f^\ast \eta \in \Nef(S)_\Q$. We assume that the limit 
\[ \overline{\sigma}_{\beta, f^\ast \eta} = \lim_{\substack{\omega \in \Amp(S)\\ \omega \to f^\ast\eta}} \sigma_{\beta, \omega} \]
exists in $\Stab(S)$. Our goal is to get a description of its heart, solely from the convergence. We will denote its slicing by $\overline{\Pslicing}_{\beta, f^\ast\eta}$, to differentiate it from the Arcara--Bertram construction.

\begin{claim}
Let $E \in \B_{\beta, V}$ be given. Then, for any $\omega \in \Amp(S)$ with $\omega^2=V$, we have that the cohomology objects of $E$ with respect to the t-structure with heart $\A_{\beta, \omega}$ are zero outside of degrees $0$ and $1$. This follows directly from $\B_{\beta, V} = \Pslicing_{\beta, \omega}((-1/2, 1/2])$, cf. Lemma \ref{lemma:heart_indep}. For instance, we have that $\O_x[-1] \in \B_{\beta, V}$, as it has phase zero.
\end{claim}

\begin{lemma} \label{lemma:idheart_sspoints}
Given $x \in S$, we claim that $\O_x \in \B_{\beta, V}[1]$ is $\overline{\sigma}_{\beta, f^\ast \eta}$-semistable of phase one. 
\end{lemma}

\begin{proof}
We argue by contradiction. Assume that $\O_x$ is not semistable, and let $B$ be its last Harder--Narasimhan factor with respect to $\overline{\sigma}_{\beta, f^\ast\eta}$, so that $Z_{\beta, f^\ast \eta}(B)$ has phase smaller than 1. This induces a short exact sequence in $\B_{\beta, V}[1]$:
\begin{equation} \label{eq:idheart_sequenceOx}
0 \to A \to \O_x \to B \to 0.
\end{equation}

Now, pick some $\omega \in \Amp(S)_\Q$ such that $\omega^2=V$. We can consider the short exact sequence above as a distinguished triangle in $\DC^b(S)$, and look at the cohomology objects with respect to the heart $\A_{\beta, \omega}$. This yields the exact sequence
\begin{equation} \label{eq:idheart_lesOx}
0 \to \HC^{-1}_{\A_{\beta, \omega}}(B) \to \HC^0_{\A_{\beta, \omega}}(A) \to \O_x \to \HC^0_{\A_{\beta, \omega}}(B) \to 0,
\end{equation}
and $\HC^{-1}_{\A_{\beta, \omega}}(A)=0$. In other words, we get $A=\HC^0_{\A_{\beta, \omega}}(A)$, hence $A \in \A_{\beta, \omega}$. 

Now, we claim that $\HC^0_{\A_{\beta, \omega}}(B)=0$. Otherwise, we will have that $\HC^0_{\A_{\beta, \omega}}(B)=\O_x$ (as $\O_x$ has no nontrivial quotients in $\A_{\beta, \omega}$), and \eqref{eq:idheart_lesOx} will give us that $A \to \O_x$ is the zero map. This is not possible, as $A \subset \O_x$ is a non-zero subobject. We get that $B=\HC^{-1}_{\A_{\beta, \omega}}(B)[1]$.

Putting all together, we get that the triangle \eqref{eq:idheart_sequenceOx} induces an exact sequence
\[ 0 \to B[-1] \to A \to \O_x \to 0 \]
in $\A_{\beta, \omega}$. Here $B[-1]$ lies both in $\B_{\beta, V}$ and in $\A_{\beta, \omega}$, and so $Z_{\beta, \omega}(B[-1])$ has argument in $(0, \pi/2]$, hence $Z_{\beta, \omega}(B)$ has argument in $\pi\cdot (1, 3/2]$. Taking the limit when $\omega \to f^\ast \eta$ contradicts the fact that $Z_{\beta, f^\ast \eta}(B)$ has phase smaller than one.
\end{proof}

\begin{cor}[cf. \cite{Bri08}*{Lemma 10.1}.] \label{cor:idheart_boundHi}
\begin{enumerate}
\item If $E \in \overline{\Pslicing}_{\beta, f^\ast \eta}(0, 1)$, then $E$ is quasi-isomorphic to a length two complex $E^{-1} \to E^0$ of locally free sheaves.
\item If $E \in \overline{\Pslicing}_{\beta, f^\ast \eta}(1)$, then $E$ is quasi-isomorphic to a length three complex $E^{-2} \to E^{-1} \to E^0$ of locally free sheaves.
\end{enumerate}
\end{cor}

\begin{proof}
The argument of \cite{Bri08} applies verbatim.
\end{proof}

\begin{lemma} \label{lemma:idheart_noH2}
If $E \in \overline{\Pslicing}_{\beta, f^\ast \eta}(1)$, then $H^{-2}(E) = 0$.
\end{lemma}

\begin{proof}
It suffices to prove it for stable objects, which we prove by contradiction. Assume that $E \in \overline{\Pslicing}_{\beta, f^\ast \eta}(1)$ is stable with $H^{-2}(E)\neq 0$. We claim that $\Hom^0(E, \O_x)=0$ for all $x$. Otherwise, we will have a short exact sequence
\[ 0 \to E \to \O_x \to \O_x/E \to 0 \]
in $\overline{\Pslicing}_{\beta, f^\ast \eta}((0, 1])$. Taking cohomology sheaves (with respect to the standard t-structure) shows that $H^{-3}(\O_x/E) = H^{-2}(E) \neq 0$, a contradiction with Corollary \ref{cor:idheart_boundHi}. It follows that $\Hom^0(E, \O_x)=0$, and so $E$ is a two-term complex $[E^{-2} \to E^{-1}]$. In particular, we get that $H^0(E)=0$.

Let us look at the spectral sequence
\[ E_2^{p, q} = \Hom^p(H^{-q}(E), \O_x) \Rightarrow \Hom^{p+q}(E, \O_x). \]
Our previous discussion ensures that $E_2^{p, q}=0$ unless $0 \leq p \leq 2$ and $1 \leq q \leq 2$. From the convergence and the fact that $\Hom^3(E, \O_x)=0$, we get the vanishing $\Hom^1(H^{-2}(E), \O_x)=0$ for all $x$. It follows that $H^{-2}(E)$ is locally free (and non-zero).

Now, pick some $y \in S$ such that $H^{-1}(E)$ is free around $y$ (possibly of rank zero!). The spectral sequence 
\[ E_2^{p, q} = \Hom^p(\O_y, H^q(E)) \Rightarrow \Hom^{p+q}(\O_y, E) \]
ensures that $\Hom(\O_y, E) \cong \Hom^2(\O_y, H^{-2}(E)) \cong \Hom(H^{-2}(E), \O_y)^\vee \neq 0$. We get a nonzero map $\O_y \to E$. But $E$ is stable of phase 1, hence we get a short exact sequence in $\Pslicing_{\beta, f^\ast\eta}((0, 1])$:
\begin{equation} \label{eq:idheart_sesforH-2}
0 \to K \to \O_y \to E \to 0.
\end{equation}

Pick some $\omega \in \Amp(S)$ with $\omega^2=V$, and let us look at the cohomology objects of \eqref{eq:idheart_sesforH-2} with respect to the heart $\A_{\beta, \omega}$. We get $\HC^{-1}_{\A_{\beta, \omega}}(K)=0$, and the exact sequence
\[ 0 \to \HC^{-1}_{\A_{\beta, \omega}}(E) \to \HC^0_{\A_{\beta, \omega}}(K) \to \O_y \to \HC^0_{\A_{\beta, \omega}}(E) \to 0 \]
in $\A_{\beta, \omega}$. As $\O_y$ is simple in $\A_{\beta, \omega}$, the last map is either zero on an isomorphism. But the second case is impossible, as we will get a non-zero map $E \to \HC^0_{\A_{\beta, \omega}}(E) =\O_y$, a contradiction. It follows that $\HC^0_{\A_{\beta, \omega}}(E)=0$, so that both $K$ and $E[-1]$ lie in $\A_{\beta, \omega}$, and \eqref{eq:idheart_sesforH-2} gives us the short exact sequence in $\A_{\beta, \omega}$:
\begin{equation} \label{eq:idheart_sesforH-2_two}
0 \to E[-1] \to K \to \O_y \to 0.
\end{equation}

At last, let us look at the cohomology sheaves of \eqref{eq:idheart_sesforH-2_two} with respect to the standard t-structure. As $H^0(E)=0$, we get the sequence
\[ 0 \to H^{-1}(E) \to H^0(K) \to \O_y \to 0 \]
in $\Coh(S)$. But $H^{-1}(E)$ is free at $y$, hence this sequence splits. This splitting, together with the fact that $H^{-1}(E) \cong H^0(K)$ is free, implies that the sequence of \eqref{eq:idheart_sesforH-2_two} admits a splitting. But this contradicts that the map $\O_y \to E$ is nonzero.
\end{proof}

\begin{remark}
The proof of Lemma \ref{lemma:idheart_noH2} differs from the one on \cite{Bri08} substantially. In the latter, one directly concludes from the fact all $\O_x$ are stable. Even knowing that one $\O_x$ is stable should be enough. Instead, our argument is significantly more involved, and the stability of (some) $\O_x$ will be checked later on. 
\end{remark}

\begin{cor}[\cite{Bri08}*{Lemma 10.1(d)}] \label{cor:idheart_torspair}
For each $E \in \overline{\Pslicing}_{\beta, f^\ast\eta}((0, 1])$, we have $H^i(E)=0$ unless $i=0, -1$. Thus, the pair of subcategories 
\[ \Ttors = \Coh(S) \cap \overline{\Pslicing}_{\beta, f^\ast \eta}((0, 1]), \quad \Ftors = \Coh(S) \cap \overline{\Pslicing}_{\beta, f^\ast \eta}((-1, 0]) \]
defines a torsion pair on $\Coh(S)$, so that $\overline{\Pslicing}_{\beta, f^\ast \eta}((0, 1])$ is the corresponding tilt.
\end{cor}

\begin{proof}
The first part is Corollary \ref{cor:idheart_boundHi} and Lemma \ref{lemma:idheart_noH2}. The argument of \cite{Bri08} applies verbatim for the second part.
\end{proof}

We recall that for every $E \in \Coh(S)$, there is a (unique) short exact sequence
\begin{equation} \label{eq:idheart_sestorspair}
0 \to Q \to E \to F \to 0
\end{equation}
of coherent sheaves, with $Q \in \Ttors$ and $F \in \Ftors$. The cohomology objects of $E$ with respect to the heart $\overline{\Pslicing}_{\beta, f^\ast\eta}((0, 1])$ can be computed from here as $\HC^0(E) = Q$, $\HC^1(E) = F[1]$,
so that \eqref{eq:idheart_sestorspair} corresponds to the triangle $\HC^0(E) \to E \to \HC^1(E)[-1] \to \ast$.

\begin{remark} \label{remark:idheart_lescohomology}
Consider a short exact sequence $0\to E_1 \to E_2 \to E_3 \to 0$ of coherent sheaves, and let us look at the induced long exact sequence of cohomology objects with respect to $\overline{\Pslicing}_{\beta, f^\ast\eta}((0, 1])$. This gives us an exact sequence
\[ 0 \to Q_1 \to Q_2 \to Q_3 \to F_1[1] \to F_2[1] \to F_3[1] \to 0 \]
in $\overline{\Pslicing}_{\beta, f^\ast\eta}((0, 1])$. In particular, we get a short exact sequence
\begin{equation} \label{eq:idheart_auxseqtorspair}
0 \to Q_1 \to Q_2 \to \ker(Q_3 \to F_1[1]) \to 0
\end{equation}
in $\overline{\Pslicing}_{\beta, f^\ast\eta}((0, 1])$. Note that the right hand side is a subobject of $Q_3$ in $\overline{\Pslicing}_{\beta, f^\ast\eta}((0, 1])$, hence it is also a coherent sheaf. From here \eqref{eq:idheart_auxseqtorspair} is also a short exact sequence of coherent sheaves, realizing $Q_2$ as an extension of $Q_1$ and a subobject (both in $\overline{\Pslicing}_{\beta, f^\ast\eta}((0, 1])$ and in $\Coh(S)$) of $Q_3$.
\end{remark}

\begin{lemma} \label{lemma:idheart_pointsonT}
Consider the torsion pair $\Ttors, \Ftors$ of Corollary \ref{cor:idheart_torspair}.
\begin{enumerate}
\item For every $x \in S$, we have $\O_x \in \Ttors$. 
\item If $F$ is a torsion sheaf on $\Ftors$, then $F$ is supported on $\Exc(f)$. 
\end{enumerate}
\end{lemma}

\begin{proof}
The basic idea is to use the torsion pair from Corollary \ref{cor:idheart_torspair}, together with the fact that $Z_{\beta, f^\ast \eta}$ maps the heart to the upper half-space.
\begin{enumerate}
\item Write $0 \to Q \to \O_x \to F \to 0$ in $\Coh(S)$, where $Q \in \Ttors$ and $F \in \Ftors$. As $\O_x$ is a simple sheaf, we have that either $Q=0$ or $F=0$, i.e. $\O_x$ is in $\Ftors$ or in $\Ttors$. But the first case will imply $\O_x[1] \in \Pslicing_{\beta, f^\ast \eta}((0, 1])$. This is a contradiction, as $Z_{\beta, f^\ast \eta}(\O_x[1])=1$ does not lie in the upper half-space.

\item Let $F$ be a torsion sheaf on $\Ftors$, so that $F[1] \in \Pslicing_{\beta, f^\ast \eta}((0, 1])$. Note first that $F$ is pure of dimension one, as $\Ftors$ is closed under subobjects and by the previous point. Now, let us compte
\[ Z_{\beta, f^\ast \eta}(F[1]) = \left( \ch_2(F)-\beta.\ch_1(F) \right) - if^\ast\eta.\ch_1(F). \]
Here $f^\ast \eta.\ch_1(F)$ must be smaller or equal than zero, as $Z_{\beta, f^\ast\eta}(F[1])$ lies on the upper half-space. But $F$ is torsion, hence $f^\ast \eta.\ch_1(F)\geq 0$. It follows that $f^\ast \eta.\ch_1(F)=0$, and so $\eta.f_\ast \ch_1(F)=0$. As $\eta$ is ample in $T$, we get that $\ch_1(F)$ is supported in $\Exc(f)$, and so is $F$. \qedhere
\end{enumerate}
\end{proof}

\begin{cor} \label{cor:idheart_sheavesonP1}
Let $i\colon C\cong \P^1 \hookrightarrow S$ be a smooth irreducible component of $\Exc(f)$ and let $d \in \Z$ be an integer. Then $i_\ast \O_C(d)$ is in $\Ttors$ (resp. $\Ftors$) if $d>\beta.C+C^2/2$ (resp. $d<\beta.C+C^2/2$).
\end{cor}

\begin{proof}
Write $0 \to Q \to i_\ast \O_C(d) \to F \to 0$, where $Q \in \Ttors$ and $F \in \Ftors$. Note that either $F=0$, $F=i_\ast \O_C(d)$, or $F$ is supported on dimension zero. But the last one cannot happen, as the elements of $\Ftors$ cannot have zero-dimensional torsion by the previous lemma.

We get that $F=0$ or $Q=0$, and so $i_\ast \O_C(d)$ or $i_\ast \O_C(d)[1]$ is in $\overline{\Pslicing}_{\beta, f^\ast\eta}((0, 1])$. Here $Z_{\beta, f^\ast\eta}(i_\ast \O_C(d)) = -(d-C^2/2) + \beta.C$, which allows us to conclude.
\end{proof}

\subsection{Heart and slopes}

Before we move on, let us relate the tilting pair $\Ttors$ and $\Ftors$ with $f^\ast\eta$-slope stability. This allows us to directly relate the heart with the standard constructions in the literature: \citelist{\cite{Bri08}*{\textsection 6} \cite{AB13}*{p. 6}} for an ample divisor, and \citelist{\cite{Tod13}*{\textsection 3.4} \cite{Tod14}*{\textsection 2.4} \cite{TX22}*{\textsection 3--4} \cite{Cho24}*{\textsection 3}} for big and nef divisors.

\begin{lemma} \label{lemma:hslope_torsIm0}
\begin{enumerate}
\item Let $Q \in \Ttors$ be an object with $\Im Z_{\beta, f^\ast\eta}(Q)=0$. We have that $Q$ is a torsion sheaf. 
\item Let $F \in \Ftors$ be an object with $\Im Z_{\beta, f^\ast\eta}(F)=0$. We have that $F$ is $f^\ast\eta$-slope semistable.
\end{enumerate}
\end{lemma}

\begin{proof}
\begin{enumerate}
\item Let us argue by contradiction. Note that we may assume that $Q$ is torsion-free, as $\Ttors$ is closed under quotients of coherent sheaves. 

We claim that $Q$ is $f^\ast\eta$-semistable. In fact, if $Q'$ is any quotient of $Q$, the fact that $\Im Z_{\beta, f^\ast\eta}(Q')\geq 0$ implies that $\mu_{f^\ast\eta}(Q') \geq \beta.f^\ast\eta = \mu_{f^\ast\eta}(Q)$. But this implies that $\Re Z_{\beta, f^\ast\eta}(Q)>0$, cf. \cite{AB13}*{p. 7}. This contradicts the fact that $Z_{\beta, f^\ast\eta}(Q)$ lies in the upper half-space.

\item The previous argument works verbatim. \qedhere
\end{enumerate}
\end{proof}

\begin{lemma} \label{lemma:hslope_ssintorsionpair}
Consider the torsion pair $\Ttors, \Ftors$ of Corollary \ref{cor:idheart_torspair}. Given $E \in \Coh(S)$, denote by $0 \to Q \to E \to F \to 0$ the short exact sequence with $Q \in \Ttors$ and $F \in \Ftors$.
\begin{enumerate}
\item Assume that $E \in \Coh(S)$ is a torsion-free $f^\ast\eta$-semistable sheaf with slope $\mu_{f^\ast\eta}(E) \leq \beta.f^\ast\eta$. Then $Q=0$ and $E=F$; in particular, $E \in \Ftors$.
\item Assume that $E \in \Coh(S)$ is a torsion-free $f^\ast \eta$-semistable sheaf with slope $\mu_{f^\ast\eta}(E)> \beta.f^\ast \eta$. Then $F$ is zero or supported on $\Exc(f)$, pure of dimension one. Moreover, we have that $Q$ is $f^\ast\eta$-semistable.
\end{enumerate}
\end{lemma}

\begin{proof}
\begin{enumerate}
\item Assume that $Q \neq 0$. Then $\mu_{f^\ast\eta}(Q)\leq \beta.f^\ast \eta$, hence $\Im Z_{\beta, f^\ast\eta}(Q)\leq 0$. The only option is that $\Im Z_{\beta, f^\ast\eta}(Q)=0$, and so $Q$ is torsion by Lemma \ref{lemma:hslope_torsIm0}. This is a contradiction with the fact that $E$ is torsion free, hence $Q=0$.

\item To start, note that $F$ must be torsion, as a direct consequence of the $f^\ast\eta$-semistability of $E$ and Lemma \ref{lemma:hslope_torsIm0}. The properties of $F$ follows from Lemma \ref{lemma:idheart_pointsonT}, and semistability of $Q$ follows directly from the semistability of $E$. \qedhere
\end{enumerate}
\end{proof}

\begin{cor} \label{cor:hslope_HNfactors}
\begin{enumerate}
\item Let $E$ be an object in $\Ftors$. Then all Harder--Narasimhan factors of $E$ with respect to $f^\ast \eta$-slope are in $\Ftors$.
\item Let $E$ be an object in $\Ttors$, and let $0 = E_0 \subseteq E_1 \subseteq \dots \subseteq E_r = E$ be its Harder--Narasimhan filtration with respect to $f^\ast\eta$-slope.

Then for every $1 \leq i\leq r$, we have that there is a subsheaf $\tilde{E}_i \subseteq E_i$ such that $E_i/\tilde{E}_i$ is torsion, supported on $\Exc(f)$, and $\tilde{E}_i$ lies in $\Ttors$. 
\end{enumerate}
\end{cor}

\begin{proof}
\begin{enumerate}
\item We induct by the length of the Harder--Narasimhan filtration of $E$, where the base case is clear.

Assume that $E \in \Ftors$ has a Harder--Narasimhan filtration of length at least two, and write $0 \to E' \to E \to E_k \to 0$, where $E_k$ is the last Harder--Narasimhan factor of $E$ with respect to $f^\ast \eta$-slope. Note that $\mu_{f^\ast\eta}(E_k)<\beta.f^\ast \eta$, as otherwise $\Im Z_{\beta, f^\ast\eta}(E[1])<0$. This way, Lemma \ref{lemma:hslope_ssintorsionpair} ensures that $E_k \in \Ftors$. Note that $E' \in \Ftors$ holds immediately (as $\Ttors, \Ftors$ is a torsion pair), which allows us to conclude by induction. 

\item Let us focus on the case $i=1$. Take $\tilde{E}_1$ to be the unique subsheaf such that $\tilde{E}_1 \in \Ttors$ and $E_1/\tilde{E}_1 \in \Ftors$, cf. \eqref{eq:idheart_sestorspair}. If $\tilde{E}_1$ is torsion, then $E_1/\tilde{E}_1 \in \Ftors$ is torsion, hence supported on $\Exc(f)$ by Lemma \ref{lemma:idheart_pointsonT}. If $\tilde{E}_1$ is torsion-free $f^\ast\eta$-semistable, we apply Lemma \ref{lemma:hslope_ssintorsionpair}. 

The general case follows by the torsion free case, and inducting on $i$. The induction step follows directly by Remark \ref{remark:idheart_lescohomology}. \qedhere
\end{enumerate}
\end{proof}

\begin{remark} \label{remark:hslope_useHNfactors}
We will mostly use this result in the following form. If $E$ is in $\Ttors$ with positive rank, and $E/E_{r-1}$ is the last Harder--Narasimhan factor, then there is a short exact sequence $0 \to \tilde{E}_{r-1} \to E \to E/\tilde{E}_{r-1} \to 0$ both in $\Coh(S)$ and in $\overline{\Pslicing}_{\beta, f^\ast\eta}((0, 1])$ (as the three factors are in $\Ttors$). Moreover, the quotient $E/\tilde{E}_{r-1}$ fits into the short exact sequence $0 \to E_{r-1}/\tilde{E}_{r-1} \to E/\tilde{E}_{r-1} \to E/E_{r-1} \to 0$ of coherent sheaves.
\end{remark}

\subsection{Proof of Theorem \ref*{teo:intro_nonexistence}} \label{subsec:pfnonexistence}

Let us briefly recall the setup of Theorem \ref{teo:intro_nonexistence}. Consider a birational map $f\colon S \to T$ from a smooth, projective surface to a normal one. Assume that for some $\eta \in \Amp(T)_\R$, $\beta \in \NS(T)_\R$, a stability condition $\overline{\sigma}_{\beta, f^\ast\eta} \in \Stab(S)$ arises as the limit of Arcara--Bertram stability conditions $\sigma_{\beta', \omega}$ as $\omega$ approaches $f^\ast\eta$. Our goal is to show that $f$ cannot have exceptional curves that are smooth of genus $g\geq 1$. We argue by contradiction, so assume $C \subset \Exc(f)$ is such a curve.

To start, we can assume that $\eta$ and $\beta$ have rational coefficients, by Bridgeland's deformation theorem. Write $\beta.C+C^2/2 = d/r$, where $d, r$ are coprime and $r>0$. As $C$ is a curve of genus $g\geq 1$, there exists a stable vector bundle $E$ on $C$ with degree $d$ and rank $r$. 

Now, consider the pushforward $i_\ast E$ under the inclusion $i\colon C \to S$, and look at the sequence \eqref{eq:idheart_sestorspair}. We can write it as $0 \to i_\ast Q \to i_\ast E \to i_\ast F \to 0$, where $i_\ast Q \in \Ttors$ and $i_\ast F \in \Ftors$. Here either $Q$ or $F$ is non-zero.

Assume that $Q \neq 0$, and denote by $d'$ and $r'$ its degree and rank, respectively. As $E$ is stable, we get that $d'/r' \leq d/r$ (note that $Q=E$ is possible). 
By Grothendieck--Riemann--Roch we have that $\ch(i_\ast Q)=0+r'[C]+(d'-C^2\cdot r'/2)[pt]$, and so 
\[ Z_{\beta, f^\ast \eta}(i_\ast Q) = \left( -d'+\frac{C^2 r'}{2} + r'\beta.C \right). \]
Using that $d'/r' \leq d/r$, it follows that $Z_{\beta, f^\ast \eta}(i_\ast Q) \geq 0$, contradicting that $i_\ast Q \in \Ttors \subset \overline{\Pslicing}_{\beta, f^\ast \eta}((0, 1])$. The same argument applies if we assume that $F \neq 0$, proving the result.

\section{Proof of Theorem \ref*{teo:intro_existence}: Pre-stability} \label{sec:prestab}

The goal of the next two sections is to prove Theorem \ref{teo:intro_existence}. To do so, let us fix some notation. We start with a birational map $f\colon S \to T$ from a smooth, projective surface to a normal surface. We assume that each connected component $C$ of $\Exc(f)$ falls into one of the two types:
\begin{itemize}
\item (ADE) $C$ is the exceptional divisor of the minimal resolution of an ADE singularity.

\item (Chain) $C=C_1 \cup \dots \cup C_r$ is a chain of rational curves, with $C_i^2 + k <0$, where $k$ is the number of curves intersecting $C_i$.
\end{itemize}
We pick $\beta \in \NS(S)_\Q$ satisfying the following restriction.

\begin{condition} \label{cond:prestab_beta}
For any connected subset of $r$ curves $C_1, \dots, C_r$ of $\Exc(f)$, we have that $\beta.\sum_{i=1}^r \delta_i C_i + (\sum_{i=1}^r C_i^2)/2$ is not an integer, where the $\delta_i$ are chosen by looking at the dual graph of the $C_i$ from the following list.
\begin{itemize}
\item Case chain: All the $\delta_i$ are $1$.

\item Case ADE: The $\delta_i>0$ are chosen such that $\sum_i \delta_i C_i$ is the \emph{fundamental cycle} of the curves, cf. \cite{Rei97}*{\textsection 4.5}.
\end{itemize}
\end{condition}

\begin{remark} \label{remark:prestab_fundcycle}
Note that the fundamental cycle has the property that if $\left(\sum_i a_i C_i \right)^2=-2$, and $a_i>0$ for all $i$, then $a_i=\delta_i$, cf. \cite{Rei97}*{p. 107}. This is the only result we will need about these values in the remainder.
\end{remark}

With this in mind, we will prove that for any $\eta \in \Amp(T)_\Q$, there is a pre-stability condition $\overline{\sigma}_{\beta, f^\ast\eta}$ on $S$, arising as a limit of the Arcara--Bertram stability conditions on $\Stab(S)$. To do so, we will apply the results of Subsection \ref{subsec:cc} by using the next technical proposition, whose proof will be given in the next two subsections.

\begin{prop} \label{prop:prestab_critstable}
Let $E \in \Coh(S)$ be a sheaf on $S$, pure of dimension one, with support connected and contained in $\Exc(f)$.
\begin{itemize}
\item (ADE) Assume that $E$ is supported on $C_1 \cup \dots \cup C_r$, where the $C_i$ are curves on the minimal resolution of an ADE singularity. If $\Hom(E, E)$ is one-dimensional, then $\ch_1(E) = \delta_1 C_1 + \dots + \delta_r C_r$, where the $\delta_i$ are as in Condition \ref{cond:prestab_beta}.

\item (Chain) Assume that $E$ is supported on $C_1 \cup \dots \cup C_r$, where $C_1, \dots, C_r$ form a chain of rational curves satisfying $C_i^2+k<0$, where $k$ is the number of curves of the chain intersecting $C_i$. Let us denote $\tilde{f}\colon S \to \tilde{T}$ the partial contraction\footnote{We recall that this contraction always exists as a projective scheme by \cite{Art62}.} of $C_1 \cup \dots \cup C_r$.

Suppose that $E$ is stable with respect to the slope
\[ \lambda_{\beta, \omega}(-) = \frac{\ch_2(-)-\beta.\ch_1(-)}{\omega.\ch_1(-)} \]
defined on $\Coh(S)_{\leq 1}$, where $\omega \in \Amp(S)$ is of the form $\tilde{f}^\ast\tilde{\eta}-\epsilon(C_1+\dots+C_r)$ for some $0 <\epsilon\ll 1$ and $\tilde{\eta} \in \Amp(\tilde{T})$. Then $\ch_1(E)=C_1 + \dots + C_r$. 
\end{itemize}
\end{prop}

This proposition has a twofold objective. On one hand, we will use it to prove the existence of pre-stability conditions $\overline{\sigma}_{\beta, f^\ast\eta}$. On the other hand, this will be used in Section \ref{sec:supp} as part of the proof of the support property.

\begin{proof}[Proof of Theorem \ref{teo:intro_existence}, pre-stability]
Let $f\colon S \to T$, $\beta \in \NS(S)_\Q$ and $\eta \in \Amp(T)_\Q$ be as in the statement. We claim that $Z_{\beta, f^\ast\eta}$ defines a (skewed) central charge on the heart\footnote{Recall from Lemma \ref{lemma:heart_indep} that $\B_{\beta, V} = \Pslicing_{\beta, \omega}((-1/2, 1/2])$ and $\Cheart_{\beta, V} = \Pslicing_{\beta, \omega}(1/2)$, where $\omega$ is ample with $\omega^2=V$.} $\B_{\beta, V}$, where $V=\eta^2$.

To show this, we will apply the criterion of Subsection \ref{subsec:cc}. Assume that there is a non-zero $E \in \Cheart_{\beta, V}$ minimal, with $\Im_{\beta, f^\ast\eta}(E)=0$. By Lemma \ref{lemma:cc_descriptionbad} we may assume that $E$ is a coherent sheaf on $S$, pure of dimension one, whose support is connected and contained in $\Exc(f)$, and with $\Hom(E, E)=\C$. Now, by assumption $\Exc(f)$ has two types of connected components.
\begin{itemize}
\item (ADE) If $\Supp(E)$ is contained in a fiber over an ADE singularity, we have that $\Supp(E)$ is supported on $C_1 \cup \dots \cup C_r$, where the $C_i$ are some of the curves in the fiber. We also have that $\Hom(E, E)=\C$. By Proposition \ref{prop:prestab_critstable}, we get that $\ch_1(E)=\delta_1 C_1 + \dots + \delta_r C_r$. An application of Grothendieck--Riemann--Roch shows that $\ch_2(E) \in \Z$, and so
\[ -\Re Z_{\beta, \omega}(E) = \ch_2(E)-\beta.(\delta_1C_1+\dots+\delta_rC_r). \]

\item (Chain) Assume that $\Supp(E)$ is contained over a fiber over a (non-ADE) chain of rational curves, say $\Supp(E) =C_1 \cup \dots \cup C_r$ (set-theoretically). Pick $\tilde{f}\colon S \to \tilde{T}$ by contracting the curves $C_1 \cup \dots \cup C_r$ to a point.

Now, note that if $\tilde{\eta} \in \Amp(\tilde{T})$ is ample in $\tilde{T}$, then $\omega = \tilde{f}^\ast\tilde{\eta}-\epsilon(C_1+\dots+C_r)$ is ample in $\tilde{S}$ for $0<\epsilon \ll 1$. In fact, we note that $\omega.C_i =\epsilon(-C_i^2+k)$, where $k$ is the number of curves of $\Exc(\tilde{f})$ intersecting $C_i$. Our assumption ensures that this is positive. By taking $\epsilon$ small enough we can ensure that $\omega^2>0$ and $\omega.\Gamma>0$ for all $\Gamma \not\subset\Exc(\tilde{f})$, hence the Nakai--Moishezon criterion applies. We pick $\tilde{\eta}$ and $\epsilon$ so that $\omega^2=\eta^2$.

This way, we have that $E \in \Cheart_{\beta, V}$ is an object in $\A_{\beta, \omega}$ that is $\sigma_{\beta, \omega}$-stable. Now, for any torsion sheaf $A \in \Coh(S)_{\leq 1}$ (which are all objects in $\A_{\beta, \omega}$) we get that $-\Re Z_{\beta, \omega}(A)/\Im Z_{\beta, \omega}(A) = \lambda_{\beta, \omega}(A)$. In particular, this shows that $E$ is stable with respect to $\lambda_{\beta, \omega}$. By Proposition \ref{prop:prestab_critstable}, this implies that $\ch_1(E)=C_1+\dots + C_r$. 

At last, another application of Grothendieck--Riemann--Roch shows that
\[ -\Re Z_{\beta, \omega}(E) = -\frac{C_1^2+\dots+C_r}{2}-r+1+d -\beta.(C_1+\dots+C_r), \]
where $d \in \Z$.
\end{itemize}
In both cases, we reach a contradiction to the fact that $\Re Z_{\beta, \omega}(E)=0$ thanks to Condition \ref{cond:prestab_beta}. This shows that no such $E$ exists, hence Proposition \ref{prop:cc_crit} applies. Thus $(Z_{\beta, f^\ast\eta}, \B_{\beta, V})$ defines a (skewed) pre-stability condition, as claimed.
\end{proof}

The rest of the section will be devoted on the proof of Proposition \ref{prop:prestab_critstable}. We will divide the proof into the two parts it encompasses.

\subsection{Criterion for ADE case}

Assume that $E \in \Coh(S)$ is a coherent sheaf on $S$, pure of dimension 1, whose support is contained on the exceptional divisor of the minimal resolution of an ADE singularity. Assume also that $\Hom(E, E)$ is one-dimensional. Write $\Supp(E)=C_1 \cup \dots \cup C_r$, where $C_i$ are all $(-2)$-curves.

\begin{claim}
We have that $\ch_1(E)^2 = -2$. To prove this, note that by assumption we have $\Hom(E, E)=\C$. Now, the fact that $\Supp(E) \subset \Exc(f)$, together with the fact that the minimal resolution of an ADE singularity is crepant, gives us
\[ \Hom^2(E, E) \cong \Hom(E, E \otimes \omega_S)^\vee = \Hom(E, E)^\vee = \C. \]
This way, we get that $\chi(E, E) = 2-\dim \Ext^1(E, E) \leq 2$, cf. \cite{Bri08}*{5.1}. On the other hand, by Hirzebruch--Riemann--Roch we have that $\chi(E, E) = -\ch_1(E)^2$. Putting all together yields $\ch_1(E)^2 \geq -2$.

But $\ch_1(E)$ is supported on the resolution of an ADE singularity. This implies that $\ch_1(E)^2$ is an even number. As we also have $\ch_1(E)^2<0$, the only option therefore is that $\ch_1(E)^2=-2$.
\end{claim}

Let us write $\ch_1(E) = \sum_{i=1}^r a_i C_i$, for some $a_i>0$. By Remark \ref{remark:prestab_fundcycle}, we get that $\sum_{i=1}^r a_i C_i$ agrees with the \emph{fundamental cycle} of the curves $C_i$. These are exactly the values of $\delta_i$ from in Condition \ref{cond:prestab_beta}.

\subsection{Criterion for chain case} \label{subsubsec:chain}

Assume $E \in \Coh(S)$ is a coherent sheaf on $S$, pure of dimension 1, supported on $C_1 \cup \dots \cup C_r$. We let $f\colon S \to \tilde{T}$ be the contraction of the $r$ curves. We assume that $E$ is stable with respect to the slope $\lambda_{\beta, \omega}$ defined on $\Coh(S)_{\leq 1}$, where $\omega \in \Amp(S)$ is of the form $\tilde{f}^\ast\tilde{\eta}-\epsilon(C_1+\dots+C_r)$ for some $0 <\epsilon\ll 1$ and $\tilde{\eta} \in \Amp(\tilde{T})$.

\begin{remark}
Note that the slope $\lambda_{\beta, \omega}$ is well defined in $\Coh(S)_{\leq 1}$, if we set $\lambda_{\beta, \omega}(A)=+\infty$ for $A$ supported on dimension zero.
\end{remark}

By assumption we have that $E$ is supported set-theoretically on $C_1 \cup \dots \cup C_r$. The key intermediate step, which is interesting on its own right, ensures that any object like this must be supported \emph{scheme-theoretically} on $C_1 \cup \dots \cup C_r$.

\begin{prop} \label{prop:chain_main}
Let $\tilde{f}\colon S \to \tilde{T}$, $\beta, \eta$, and $\omega$ be as before. Let $A$ be a $\lambda_{\beta, \omega}$-stable sheaf, pure of dimension 1, and whose support equals $\Exc(\tilde{f})$ set-theoretically. Then the scheme-theoretic support of $A$ is reduced, and $A$ is the pushforward of a line bundle on $C_1 \cup \dots \cup C_r$.
\end{prop}

The proof relies in two important facts. The first one is a computation of $Lj^\ast j_\ast$ for an effective Cartier divisor on a variety. This is a well-known result, which we reproduce to emphasize that the statement holds even if the divisor is non-reduced.

\begin{lemma}[cf. \cite{Huy06}*{11.4}] \label{lemma:chain_triangleadjunction}
Let $X$ be a smooth, projective variety and let $j\colon Y \to X$ be an effective Cartier divisor. Given any coherent sheaf $E \in \Coh(Y)$, there is an exact triangle $E \otimes_Y \O_X(-Y)|_Y[1] \to Lj^\ast j_\ast E \to E \to \ast$.
\end{lemma}

The second key result we will use is a description of indecomposable vector bundles over chains of rational curves. This mimics the fact that every vector bundle on $\P^1$ splits as a direct sum of line bundles.

\begin{teo}[\cite{DG01}*{2.7}] \label{teo:chain_tf}
Let $C=C_1 \cup \dots \cup C_r$ be a reduced curve, where each $C_i$ is a smooth rational curve, and such that each $C_i$ intersect only $C_{i\pm 1}$ in exactly one node. Assume that $E$ is a pure dimension 1 sheaf on $S$, supported on all on $C$, and indecomposable. Then $E$ is a line bundle on $S$.
\end{teo}

\begin{proof}[Proof of Proposition \ref{prop:chain_main}]
We will argue by contradiction. Assume that such $A$ is not supported on $C=C_1 \cup \dots \cup C_r$ scheme-theoretically, namely $\O_S(-C)\cdot A \neq 0$. Let $k\geq 1$ be the maximum value such that $\O_S(-kC) \cdot A \neq 0$. If $i\colon C' \hookrightarrow S$ denotes the $(k-1)$th order thickening, our assumption implies that $A$ fits in an extension
\begin{equation} \label{eq:chain_filtration_A}
0 \to i_\ast F \to A \to i_\ast G \to 0,
\end{equation}
where $F, G \in \Coh(C')$, and $A$ is \emph{not} an extension coming from $C'$. 

Let us denote by $\xi \in \Ext^1_S(i_\ast G, i_\ast F)$ the extension class of \eqref{eq:chain_filtration_A}. The assumption of $A$ not being an extension on $C'$ guarantees that $\xi$ is not in the image of
\begin{equation} \label{eq:chain_ext1adjunction}
\Ext^1_{C'}(G, F) \to \Ext^1_S(i_\ast G, i_\ast F) \cong \Ext^1_{C'}(Li^\ast i_\ast G, F).
\end{equation}
Note that this adjunction map is the map from Lemma \ref{lemma:chain_triangleadjunction}. This way, we consider the triangle $G \otimes_{C'} \O_S(-C')|_{C'}[1] \to Li^\ast i_\ast G \to G \to \ast$ and apply $\Hom_{C'}(-, F)$ to it. This gives us the exact sequence
\[ \Ext^1_{C'}(G, F) \to \Ext^1_{C'}(Li^\ast i_\ast G, F) \to \Hom_{C'}(G \otimes_{C'} \O_S(-C'|_{C'}), F). \]
As $\xi$ is not in the image of \eqref{eq:chain_ext1adjunction}, we get that 
\[ 0 \neq \Hom_{C'}(G \otimes_{C'} \O_S(-C')|_{C'}, F) = \Hom_S(i_\ast G \otimes \O_S(-C'), i_\ast F). \]
From \cite{HL10}*{1.3.3} we get the inequality
\begin{equation} \label{eq:chain_firstinequality}
\lambda_{\beta, \omega, \mathrm{min}}(i_\ast G \otimes \O_S(-C')) \leq \lambda_{\beta, \omega, \mathrm{max}}(i_\ast F).
\end{equation}

On the other hand, we have that $i_\ast F$ is a subobject of $E$, and so $\lambda_{\beta, \omega, \mathrm{max}}(i_\ast F) < \lambda_{\beta, \omega}(E)$ (as $E$ is stable). Similarly, $\lambda_{\beta, \omega}(E) < \lambda_{\beta, \omega, \mathrm{min}}(i_\ast G)$. These two inequalities, together with \eqref{eq:chain_firstinequality}, yield the result
\begin{equation} \label{eq:chain_secondinequality}
\lambda_{\beta, \omega, \mathrm{min}}(i_\ast G \otimes \O_S(-C')) < \lambda_{\beta, \omega, \mathrm{min}}(i_\ast G).
\end{equation}

At last, let $i_\ast G \to Q \to 0$ be the last factor of the Harder--Narasimhan filtration of $Q$. Here $Q$ is pure of dimension one and supported set-theoretically on $C'$. On the other hand, we claim that the Harder--Narasimhan filtration of $i_\ast G \otimes \O_S(-C')$ has last factor $Q \otimes \O_S(-C')$. 

In fact, given a sheaf $B$ supported on $C$ we have that $\lambda_{\beta, \omega}(B \otimes \O_S(-C')) - \lambda_{\beta, \omega}(B) = k/\epsilon$ is constant. Together with \eqref{eq:chain_secondinequality}, we get a contradiction. This proves that $E$ must be the pushforward of a coherent sheaf on $C$ with the reduced structure. 

The second part follows from the fact that $E$ is indecomposable, as it is a stable sheaf. This way $E$ is a line bundle on $C_1 \cup \dots \cup C_r$ thanks to Theorem \ref{teo:chain_tf}.
\end{proof}

\section{Proof of Theorem \ref*{teo:intro_existence}: Support property} \label{sec:supp}

Let us consider the setup of Section \ref{sec:prestab}, where we constructed the pre-stability condition $\overline{\sigma}_{\beta, f^\ast\eta}$. The goal of this section is to prove the support property for $\overline{\sigma}_{\beta, f^\ast\eta}$, which is the remaining piece towards the proof of Theorem \ref{teo:intro_existence}.

Our approach follows the ideas of \citelist{\cite{Tod13}*{\textsection 3.7} \cite{TX22}*{\textsection 6}}. First, we prove a weak support property in Lemma \ref{lemma:weaksupp_main}, which will allow us to construct stability conditions $\overline{\sigma}_{\beta, sf^\ast \eta}$ as $s$ approaches infinity. Second, we give a description of objects that are $\overline{\sigma}_{\beta, sf^\ast \eta}$-stable as $s$ goes to infinity in Lemma \ref{lemma:bridgie_main}. 

The remainder of the section is devoted to providing bounds to each of these objects. As we mentioned in the introduction, we will prove the support property only for $\beta \in \NS(S)_\Q$ satisfying certain conditions (see Condition \ref{cond:bounds_smallkij}). We will discuss various tools for bounding the Chern numbers of the stable objects in Subsection \ref{subsec:bounds}. These bounds will be used to conclude in Subsection \ref{subsec:wrapping}

\subsection{Weak support property} \label{subsec:weaksupp}

Let us start by recalling the following version of the Bogomolov--Gieseker inequality for semistable torsion-free sheaves, with respect to a big and nef class.

\begin{prop}[\cite{GKP16}*{5.1}]
Let $F$ be a torsion-free $f^\ast\eta$-semistable sheaf on $S$. Then we have $\ch_1^\beta(F)^2 \geq 2\ch_0(F) \ch_2^\beta(F)$.
\end{prop}

With this in mind, we will prove the following inequality.

\begin{lemma}[cf. \cite{Tod13}*{3.18}] \label{lemma:weaksupp_main}
Let $E \in \overline{\Pslicing}_{\beta, f^\ast \eta}((0, 1])$ be a $\sigma_{\beta, f^\ast\eta}$-semistable object. Then the inequality $(\ch_1^\beta(E).f^\ast \eta)^2 \geq 2\eta^2 \ch_0(E)\ch_2^\beta(E)$ holds.
\end{lemma}

\begin{proof}
The proof of \cite{Tod13}*{3.18} works almost verbatim, replacing the use of \cite{Tod13}*{3.9} with the description of Remark \ref{remark:hslope_useHNfactors}.
\end{proof}

\begin{cor}[cf. \cite{Tod13}*{3.19}]
Let $\beta, \eta$ be as before. There exists a (unique) continuous map\footnote{We are abusing notation here, as our definition of $\Stab(S)$ requires the (full) support property. However, the weak support property allows us to show that $\sigma_{\beta, f^\ast \eta}$ is a stability condition with respect to a restricted lattice. This way, the deformations $\sigma_s$ are stability conditions with respect to this lattice. See for instance the discussion on \cite{Tod13}*{\textsection 3.6} for more details.} $\R_{>0} \to \Stab(S)$, $s \mapsto \sigma_s$ so that $\sigma_1 = \overline{\sigma}_{\beta, f^\ast \eta}$, and such that the central charge of $\sigma_s$ is $Z_{\beta, sf^\ast \eta}$. The stability conditions constructed in this fashion all have heart $\overline{\Pslicing}_{\beta, f^\ast\eta}((0, 1])$. In particular, for any fixed numerical class, the $\sigma_s$ admit a locally finite wall and chamber decomposition.
\end{cor}

\subsection{Bridgeland--to--Gieseker comparison}

The next step towards proving the support property is to compare objects that are $\sigma_s$-semistable for $s \gg 0$ with $f^\ast\eta$-semistability.

\begin{lemma}[cf. \citelist{\cite{Tod13}*{Lemma 3.21} \cite{TX22}*{Lemma 6.4}}] \label{lemma:bridgie_main}
Let $E \in \overline{\Pslicing}_{\beta, f^\ast \eta}((0, 1])$ be given. Assume that $E$ is $\sigma_s$-semistable for all $s$ sufficiently large. Then $E$ has one of the following forms.
\begin{enumerate}[label=(F\arabic*)]
\item $H^{-1}(E)=0$, and $E=H^0(E)$ is a coherent sheaf, $f^\ast\eta$-slope semistable, whose torsion subsheaf lies in\footnote{Here $\Ftors$ and $\Ttors$ are the corresponding categories for $\overline{\Pslicing}_{\beta, f^\ast\eta}$ described in Corollary \ref{cor:idheart_torspair}.} $\Ftors$. 
\item $H^{-1}(E)=0$, and $E=H^0(E)$ is torsion, with phase $\neq 1$.
\item $H^{-1}(E)$ and $H^0(E)$ are torsion, supported on $\Exc(f)$ union finitely many points. 
\item $H^0(E)$ is a torsion sheaf, supported on $\Exc(f)$ union finitely many points, and $H^{-1}(E)$ is a coherent sheaf, $f^\ast\eta$-slope semistable of slope $\beta.f^\ast\eta$, whose torsion subsheaf lies in $\Ftors$.
\item $H^0(E)$ is a torsion sheaf, supported on $\Exc(f)$ union finitely many points, and $H^{-1}(E)$ is a torsion-free coherent sheaf, $f^\ast\eta$-slope semistable (of slope smaller than $\beta.f^\ast\eta$). 
\end{enumerate}
The objects in (F3) and (F4) have phase 1 for all $s$, while the other three have phase smaller than 1 for all $s$. 
\end{lemma}

The proof will take the remainder of this subsection. The key idea is that to check $\sigma_s$-semistability of an object $A \in \overline{\Pslicing}_{\beta, f^\ast \eta}((0, 1])$, we can compute the slope $\psi_s(A)=-\Re Z_{\beta, sf^\ast \eta}(A)/\Im Z_{\beta, sf^\ast \eta}(A)$. We have
\begin{equation} \label{eq:presupp_auxslope}
\psi_s(A) = -\frac{s\eta^2}{2}(\mu_{f^\ast \eta}(A)-\beta.f^\ast \eta)^{-1} + O(1/s),
\end{equation}
provided that $\ch_0(A)\neq 0$. This will relate $\sigma_s$ and Gieseker semistability. 

Through this subsection, we let $E$ be a $\sigma_s$-semistable object, for $s \gg 0$. We will divide in various cases depending on the rank of $H^0(E)$ and $H^{-1}(E)$.

\subsubsection{Case I} \label{subsubsec:caseI}

To start, let us assume that $E$ satisfies that both $\ch_0(H^0(E))$ and $\ch_0(H^{-1}(E))$ are greater than zero. From the exact sequence $0 \to H^{-1}(E)[1] \to E \to H^0(E) \to 0$ in $\overline{\Pslicing}_{\beta, f^\ast\eta}((0, 1])$, we get that $\psi_s(H^{-1}(E)[1]) \leq \psi_s(H^0(E))$ for all $s \gg 0$. Using \eqref{eq:presupp_auxslope}, we get that
\[ -(\mu_{f^\ast\eta}(H^{-1}(E)[1])-\beta.f^\ast\eta)^{-1} \leq -(\mu_{f^\ast\eta}(H^0(E))-\beta.f^\ast\eta)^{-1}. \]
Now, note that $\Im Z_{\beta, f^\ast \eta}(H^0(E))\geq 0$, hence the last term is either $+\infty$ (if the imaginary part is zero) or negative. Similarly, we get that the first term is either $+\infty$ or positive. This is a contradiction unless $\Im Z_{\beta, f^\ast\eta}(H^0(E))=0$. But in this case $\ch_0(H^0(E))=0$ by Lemma \ref{lemma:hslope_torsIm0}. This contradicts our original assumption that both $\ch_0(H^0(E))$ and $\ch_0(H^{-1}(E))$ are non-zero. It follows that at least one of them must be a torsion sheaf.

\subsubsection{Case II}

Assume now that $\ch_0(H^0(E))>0$, Note that the previous case ensures that $\ch_0(H^{-1}(E))=0$; let us prove that $H^{-1}(E) = 0$ by contradiction. We already know that $H^{-1}(E)$ must be torsion, and Lemma \ref{lemma:idheart_pointsonT} ensures that $H^{-1}(E)$ is supported on $\Exc(f)$. These two facts imply that $\psi_s(H^{-1}(E)[1])=+\infty$. But then the short exact sequence $0 \to H^{-1}(E)[1] \to E \to H^0(E) \to 0$
in $\overline{\Pslicing}_{\beta, f^\ast\eta}((0, 1])$, together with the $\sigma_s$-semistability of $E$, gives us that $\psi_s(H^0(E))=+\infty$, i.e. $\Im Z_{\beta, f^\ast\eta}(H^0(E))=0$. Once again, this is a contradiction by Lemma \ref{lemma:hslope_torsIm0}. Therefore, we must have $H^{-1}(E)=0$. 

Now, note that $E=H^0(E)$ is $f^\ast\eta$-slope semistable. Otherwise, we can take the short exact sequence of Remark \ref{remark:hslope_useHNfactors}:
\[ 0 \to \tilde{E}_{r-1} \to E \to E/\tilde{E}_{r-1} \to 0, \]
which is a short exact sequence in $\overline{\Pslicing}_{\beta, f^\ast\eta}((0, 1])$. Here $E/\tilde{E}_{r-1}$ has the same $f^\ast\eta$-slope as $E/E_{r-1}$. Using that $E$ is $\sigma_s$-semistable for all $s \gg 0$ together with \eqref{eq:presupp_auxslope} shows that $\mu_{f^\ast\eta}(E)\leq \mu_{f\ast\eta}(E/E_{r-1})$, hence $E$ is $f^\ast\eta$-semistable.

So far, we have shown that $H^0(E)$ is $f^\ast\eta$-slope semistable. We now claim that $H^0(E)_{tors} \in \Ftors$. To prove this, consider the short exact sequence of sheaves
\[ 0 \to H^0(E)_{tors} \to E \to H^0(E)_{tf} \to 0, \]
and look at the long exact sequence of cohomology objects with respect to the heart $\overline{\Pslicing}_{\beta, f^\ast\eta}((0, 1])$. Here $E$ and $H^0(E)_{tf}$ lie in the heart, as $E \in \Ttors$. We get
\[ 0 \to \HC^0(H^0(E)_{tors}) \to E \to H^0(E)_{tf} \to \HC^1(H^0(E)_{tors}) \to 0. \]
Here $\HC^0(H^0(E)_{tors})$ is a torsion sheaf, which would destabilize $E$ if non-zero (as $E$ does not have phase 1). It follows that $E$ satisfies (F1).

\subsubsection{Case III}

Assume now that $\ch_0(H^{-1}(E))>0$, so that $\ch_0(H^0(E))=0$ by Case I, \ref{subsubsec:caseI}. The short exact sequence $0 \to H^{-1}(E)[1] \to E \to H^0(E) \to 0$ in $\overline{\Pslicing}_{\beta, f^\ast\eta}((0, 1])$ gives us that $\phi_s(E) \leq \phi_s(H^0(E))$ for all $s \gg 0$, provided that $H^0(E) \neq 0$. This implies that $f^\ast\eta.\ch_1(H^0(E))=0$, as otherwise \eqref{eq:presupp_auxslope} gives a contradiction. Thus, $H^0(E)$ is supported on $\Exc(f)$ up to finitely many points. 

Note from Corollary \ref{cor:hslope_HNfactors} that all the Harder--Narasimhan factors of $H^{-1}(E)$ lie in $\Ftors$. As before, this implies that $H^{-1}(E)$ is $f^\ast\eta$-semistable. If $H^{-1}(E)$ has $f^\ast\eta$-slope $\beta.f^\ast\eta$, we get that $E$ is in (F4).

At last, assume that $H^{-1}(E)$ has slope smaller than $\beta.f^\ast\eta$. The torsion subsheaf of $H^{-1}(E)$, if not zero, would give a destabilizing subobject of $E$ for $s \gg 0$. It follows that $H^{-1}(E)$ is torsion-free, hence $E$ satisfies (F5).

\subsubsection{Case IV}

We assume now that $\ch_0(H^0(E))=0$ and $\ch_0(H^{-1}(E)) =0$. Here $H^{-1}(E)$ is supported on $\Exc(f)$ by Lemma \ref{lemma:idheart_pointsonT}, and so it has $\sigma_s$-phase 1 (or $H^{-1}(E)=0$). The short exact sequence
\[ 0 \to H^{-1}(E)[1] \to E \to H^0(E) \to 0 \]
in $\overline{\Pslicing}_{\beta, f^\ast\eta}((0, 1])$ implies that either $H^{-1}(E)=0$ or $E$ is $\sigma_s$-semistable of phase $1$.

If $H^{-1}(E)=0$, then $E$ is torsion, falling into (F2). Otherwise, we have that $H^0(E)$ must also have phase 1, and it is torsion as $\ch_0(H^0(E))=0$. Thus $0 = \Im Z_{\beta, f^\ast\eta}(H^0(E)) = \ch_1(H^0(E)).f^\ast\eta$ implies that $H^0(E)$ is supported on finitely many points and $\Exc(f)$. This falls into (F3).

\subsection{Bounds} \label{subsec:bounds}

In this subsection we will give the bounds that will help us prove the support property. As we mentioned in the introduction, we are only able to prove the support property when the connected components of $\Exc(f)$ are chains of curves (i.e. we do not allow $T$ to have singularities of type $D_n, n \geq 4$ or $E_6, E_7, E_8$), and for certain values of $\beta$. For this we will need to introduce some notation.

Denote by $C_1, \dots, C_\ell$ the connected components of $\Exc(f)$. Each one is a chain, say $C_i = C_{i,1} \cup \dots \cup C_{i, r_i}$ with $C_{i,j}^2 = -n_{i,j}$. The $C_{i,j}$ form the basis of a negative definite subspace of $\NS(S)_\R$ by the negativity lemma, cf. \cite{Mum61}*{p. 6}. This way, we fix a norm $\norm{\cdot}$ on $\NS(S)_\R$, such that its restriction on $\bigoplus_{i,j}\R C_{i,j}$ is the negative of the intersection product:
\begin{equation} \label{eq:bounds_norm}
\norm{\sum_{i,j} a_{i,j}C_{i,j}} = -\left( \sum_{i,j} a_{i,j}C_{i,j} \right)^2.
\end{equation}
We also fix a number $M>0$ such that $\norm{\sum_{i,j} a_{i,j}C_{i,j}} \leq M \sum_{i,j} \abs{a_{i,j}}$ for all $a_{i,j} \in \R$.

By Condition \ref{cond:prestab_beta} we have that $\beta.C_{i,j}-n_{i,j}/2$ is not an integer. This way, we fix integers $k_{i,j}$ such that
\begin{equation} \label{eq:bounds_kij}
k_{i,j}-1 < \beta.C_{i,j} - \frac{n_{i,j}}{2} < k_{i,j}.
\end{equation}
Moreover, we are also assuming that $\beta.(C_{i, j}+\dots+C_{i, j})-(n_{i, j}+\dots +n_{i, j'})/2$ is not an integer for all $1 \leq j\leq j' \leq r_i$. Using \eqref{eq:bounds_kij} we get the bound
\begin{equation} \label{eq:bounds_presmallkij}
(k_{i,j}+\dots+k_{i,j'})-(j'-j+1) < \beta.(C_{i, j}+\dots+C_{i, j'})-\frac{n_{i, j}+\dots +n_{i, j'}}{2}.
\end{equation}
We will impose the following additional condition on $\beta$.

\begin{condition} \label{cond:bounds_smallkij}
In the previous notation, assume that
\[ \beta.(C_{i, j}+\dots+C_{i, j'})-\frac{n_{i, j}+\dots +n_{i, j'}}{2}<(k_{i,j}+\dots+k_{i,j'})-(j'-j) \]
for all $1 \leq j\leq j' \leq r_i$. 
\end{condition}

\begin{remark}
Note that this condition is non-vacuous. In fact, set $\beta.C_{i,j}=k_{i,j}+n_{i,j}/2-1+\epsilon$ for $\epsilon$ small enough. This defines an element $\beta \in \bigoplus_{i,j} \R C_{i,j} \subset \NS(S)$ satisfying Condition \ref{cond:bounds_smallkij}.
\end{remark}

We start with the following version of \cite{TX22}*{Lemma 6.5} that we will use to bound intersections of curves outside of $\Exc(f)$.

\begin{lemma} \label{lemma:bounds_effectivenotExc}
There exists a constant $A>0$, depending only on the ample class $\eta$, such that for any effective divisor $D$ without components in $\Exc(f)$, we have (i) $\norm{D} \leq \sqrt{A/\eta^2} (D.f^\ast\eta)$, (ii) $D.C_{i,j} \leq \sqrt{\frac{A}{\eta^2}}(D.f^\ast\eta)$, and (iii) $D^2 + \frac{A}{\eta^2}(D.f^\ast\eta)^2 \geq 0$.
\end{lemma}

\begin{proof}
Pick numbers $\delta_{i,j}>0$ small enough such that $f^\ast\eta-\sum_{i,j}\delta_{i,j}C_{i,j}$ is ample. This way, we have that
\[ \norm{D} \leq \sqrt{\frac{A}{\eta^2}}(D.(f^\ast\eta-\sum_{i,j}\delta_{i,j}C_{i,j})) \]
for all $D$ effective, cf. \cite{Laz04}*{1.4.29}. The first result follows immediately as $D.C_{i,j}\geq 0$. The other two follow from here directly, as $\abs{D.D'}/(\norm{D}\norm{D'})$ is bounded.
\end{proof}

The next lemma will bound the objects of phase one supported on $\Exc(f)$.

\begin{lemma} \label{lemma:bounds_onExc}
Let $E \in \overline{\Pslicing}_{\beta, f^\ast\eta}((0, 1])$ be an object supported on $\Exc(f)$. Then there exists a constant $\delta>0$ such that $\delta^2\ch_1(E)^2 +  (\ch_2^\beta(E))^2 \geq 0$.  
\end{lemma}

\begin{proof}
It suffices to prove this for a minimal object $E \in \overline{\Pslicing}_{\beta, f^\ast\eta}((0, 1])$. Such object will be either $E=H^0(E) \in \Ttors$ or $E=H^{-1}(E)[1] \in \Ftors[1]$. Note also that $\Supp E$ is connected, so it is contained either on an ADE fiber (which are only chains of $(-2)$-curves now), or on a chain fiber.

\begin{itemize}
\item (ADE) Assume $E$ is supported on a chain of $(-2)$-curves. Note that $\Hom(E, E)$ is one-dimensional by the minimality. This way we get that $\ch_1(E) = \pm(C_{i,j}+\dots+C_{i,j'})$ for some $1\leq j\leq j' \leq r_i$, by Proposition \ref{prop:prestab_critstable}.

\item (Chain) Assume that $E$ is supported on a chain of curves, say on $C_{i, j} \cup \dots \cup C_{i, j'}$. Let us assume that $E=H^0(E)$; the other case is handled analogously.

We claim that $F \in \Ftors$ for all $F \subsetneq E$. In fact, if $F \subsetneq E$ is not in $\Ftors$, it has a subsheaf in $F' \in \Ttors$. This way, the short exact sequence
\[ 0 \to F' \to E \to E/F' \to 0 \]
of coherent sheaves is also a short exact sequence on $\overline{\Pslicing}_{\beta, f^\ast\eta}((0, 1])$. This contradicts the minimality of $E$. 

In particular, we have that for any non-trivial subsheaf $F \subset E$ that $\ch_2^\beta(F)<0$, while $\ch_2^\beta(E)>0$.. This way, given any ample divisor $\omega$, we have that $E$ is stable with respect to the slope $\lambda_{\beta, \omega}$, in the notation of Proposition \ref{prop:prestab_critstable}. It follows that $\ch_1(E)=C_{i,j} + \dots + C_{i,j'}$, as claimed. 
\end{itemize}
In any case, we get that the coefficients of $\ch_1(E)$ have a finite list of options. Similarly, we have that $\ch_2^\beta(E)$ is bounded below, as $\ch_2^\beta(E)>0$ and $\ch_2(E) \in \frac{1}{2}\Z$. Putting these two results together gives us the claimed bound.
\end{proof}

\begin{cor} \label{cor:bounds_onExc}
Assume $E$ is supported on $\Exc(f)$. If $E \in \Ttors$, then $\ch_2^\beta(E) \geq \delta \norm{\ch_1^\beta(E)}$; if $E \in \Ftors$, then $-\ch_2^\beta(E) \geq \delta \norm{\ch_1^\beta(E)}$.
\end{cor}

\begin{proof}
Follows directly from the choice of $\norm{\cdot}$ in \eqref{eq:bounds_norm} and the sign of $\ch_2^\beta(E)$. 
\end{proof}

So far, we have not used the restriction on $\beta$ given in Condition \ref{cond:bounds_smallkij}. We will use it to prove the following lemma, which will help us to establish our last bound.

\begin{lemma} \label{lemma:bounds_objectsinTF}
Pick $1 \leq i \leq \ell$ and $1 \leq j \leq j' \leq r_i$. We have that the sheaf\footnote{We write $C_{i,j,j'} = C_{i,j} \cup C_{i,j+1} \cup \dots \cup C_{i,j'}$ for the sake of readability.} $\O_{C_{i, j, j'}}(s_j, s_{j+1}, \dots, s_{j'})$ lies in $\Ftors$ if $s_a \leq k_a$ for all $j \leq a \leq j'$, with at least one strict inequality. We also have that $\O_{C_{i, j, j'}}(k_j, k_{j+1}, \dots, k_{j'}) \in \Ttors$. 
\end{lemma}

\begin{proof}
The first part of the result follows by induction on $j'-j$, where the base case is Corollary \ref{cor:idheart_sheavesonP1}. For the induction step pick an entry $a$ such that $s_a<k_a$, and look at the exact sequence\footnote{We are assuming $j<a<j'$; a small modification applies when $a=j$ or $a=j'$.}
\begin{align*}
0 \to& \O_{C_{i,j,a-1}}(s_j, \dots, s_{a-2}, s_{a-1}-1) \oplus \O_{C_{i, a+1,j'}}(s_{a+1}-1, s_{a+2}, \dots, s_{j'}) \\
&\to \O_{C_{i,j,j'}}(s_j, \dots, s_{j'}) \to \O_{C_{i,a}}(s_a) \to 0.  
\end{align*}
The first term lies in $\Ftors$ by induction, and the second one lies in $\Ftors$ as $s_a \leq k_a-1$ (and by the base case). This proves the required result. 

For the second part, note that
\[ \ch_2^\beta(\O_{C_{i,j,j'}}(k_j, \dots, k_{j'})) = \sum_{a=j}^{j'}\left( k_a+\frac{n_{i,a}}{2}\right)-(j'-j)-\beta.(C_{i,j}+\dots + C_{i,j'}) >0 \]
by Condition \ref{cond:bounds_smallkij}. This shows that $\O_{C_{i,j,j'}}(k_j, \dots, k_{j'})$ cannot lie in $\Ftors$. Not only that, but this also says that no pure dimension one quotient lies in $\Ftors$. The only option is that $\O_{C_{i,j,j'}}(k_j, \dots, k_{j'}) \in \Ttors$, as $\Ttors, \Ftors$ is a torsion pair of $\Coh(S)$. 
\end{proof}

As a consequence we will prove our final bound, cf. \cite{Tod13}*{3.7}.

\begin{lemma} \label{lemma:bounds_ch1}
Let $E \in \overline{\Pslicing}_{\beta, f^\ast\eta}((0, 1))$ be an object without components of phase 1. Write $\ch_1(H^0(E)_{tors}) = \sum_{i,j} b_{i,j}C_{i,j} +\Gamma$, where $\Gamma$ is effective and without components in $\Exc(f)$. Given $1 \leq i \leq \ell$ and $1 \leq j \leq j' \leq r_i$, we have that
\begin{gather*}
- (C_{i,j}+\dots+C_{i,j'}).\ch_1(E) + \left( \sum_{a=j}^{j'}(k_{i,a}+n_{i,a})-(2j'-2j+1)\right)\ch_0(E) \\
\leq \ch_0(H^0(E)) + b_{i, j} + \sum_{a=j}^{j'}(n_{i, a}-2)(\ch_0(H^0(E)) + b_{i, a}),
\end{gather*} 
and similarly by replacing $b_{i,j}$ with $b_{i,j'}$. 
\end{lemma}

\begin{proof}
Let us compute $\chi(\O_{C_{i,j,j'}}(k_{i,j}, \dots, k_{i,j'}), E)$ in two different ways. First, by Hirzebruch--Riemann--Roch it equals
\[ - (C_{i,j}+\dots+C_{i,j'}).\ch_1(E) + \left( \sum_{a=j}^{j'}(k_{i,a}+n_{i,a})-(2j'-2j+1)\right)\ch_0(E). \]
Second, we can write
\begin{equation} \label{eq:bounds_eulerchar}
\begin{aligned}
&\chi(\O_{C_{i,j,j'}}(k_{i,j}, \dots, k_{i,j'}), E) \\ =& \sum_{s=1}^\infty (-1)^s \dim \Hom^s(\O_{C_{i,j,j'}}(k_{i,j}, \dots, k_{i,j'}), E),
\end{aligned}
\end{equation}
where we used that both object lie in the heart of a t-structure together with the assumption that $E$ has no factor in phase 1. To bound the remaining terms we use Serre duality:
\[ = \sum_{s=1}^\infty (-1)^s \dim \Hom^{2-s}(E, \O_{C_{i,j,j'}}(k_{i,j}+n_{i,j}-2, \dots, k_{i,j'}+n_{i,j'}-2)) \]
We have to be careful here, as the coherent sheaf $\O_{C_{i,j,j'}}(k_{i,j}+n_{i,j}-2, \dots, k_{i,j'}+n_{i,j'}-2)$ will not be in general in the heart of our t-structure. Instead, we have short exact sequences
\begin{equation} \label{eq:bounds_keyses}
\begin{gathered}
0 \to \O_{C_{i,j,j'}}(k_{i,j}-1, k_{i, j+1} \dots, k_{i,j'}) \to \O_{C_{i,j,j'}}(k_{i,j}+n_{i,j}-2, \dots, \\
k_{i,j'}+n_{i,j'}-2) \to \O_Z \to 0,
\end{gathered}
\end{equation}
where $Z$ is a reduced zero-dimensional subscheme with $n_{i,j}-1$ points in $C_{i,j}$, and $n_{i,a}-2$ points on the remaining curves. The first term now lies in $\Ftors$, while the last one lies in $\Ttors$. Applying $\Hom^t(E, -)$ to \eqref{eq:bounds_keyses} shows that $\Hom^t(E, \O_{C_{i,j,j'}}(k_{i,j}+n_{i,j}-2, \dots, k_{i,j'}+n_{i,j'}-2))$ is zero for $t \leq 0$ and 
\[ \dim \Hom(E, \O_{C_{i,j,j'}}(k_{i,j}+n_{i,j}-2, \dots, k_{i,j'}+n_{i,j'}-2)) \leq \dim \Hom(E, \O_Z). \]
Plugging these into \eqref{eq:bounds_eulerchar} (together with the trivial bound on the $s=1$ term) yields
\begin{equation} \label{eq:bounds_eulerchartwo}
\chi(\O_{C_{i,j,j'}}(k_{i,j}, \dots, k_{i,j'}), E) \leq \dim \Hom(E, \O_Z).
\end{equation}
It remains to bound $\Hom(E, \O_Z) = \Hom(H^0(E), \O_Z)$. To do so, we write the short exact sequence $0 \to H^0(E)_{tors} \to H^0(E) \to H^0(E)_{tf} \to 0$ of coherent sheaves. If $p \in C_{i,a}$ is a general point (avoiding all other curves of $\Supp(H^0(E)_{tors})$, and contained in the locus where $H^0(E)_{tf}$ is free), then the bound $\dim\Hom(H^0(E), \O_p) \leq \ch_0(H^0(E)) + b_{i, a}$ follows. Using these to bound $\dim \Hom(E, \O_Z)$ and plugging them in \eqref{eq:bounds_eulerchartwo} gives us the claimed inequality. 
\end{proof}

We will use this lemma to bound some of the objects coming from Lemma \ref{lemma:bridgie_main}. We write $N_1=\frac{1}{2}\sum_{i,j}\sum_{a=j}^{r_i} n_{i,a}$, $P=-\left( \sum_{i,j}\sum_{a=j}^{r_i}(C_{i,1}+\dots+C_{i,a}) \right)^2$, $L=\sum_{i=1}^\ell \sum_{j=1}^{r_i} j^2$, and $N_5=\sum_{i,j} \sum_{j'=j}^{r_i}\left( -j'+1 + \sum_{a=1}^{j'} \frac{n_{i, a}}{2} \right)$.

\begin{cor} \label{cor:bounds_F125}
\begin{enumerate}
\item[(F1)] Assume that $E \in \overline{\Pslicing}_{\beta, f^\ast\eta}((0, 1))$ has $H^{-1}(E)=0$, $r=\ch_0(E)>0$, and that the torsion subsheaf $E_{tf}$ of $E$ lies in $\Ftors$. Write $\ch_1(E_{tors})=\sum_{i,j}b_{i,j}C_{i,j}$ and $\ch_1^\beta(E_{tf})=f^\ast D + \sum_{i,j} a_{i,j} C_{i,j}$. Then $\norm{\sum_{i,j}b_{i,j}C_{i,j}} \leq PM\norm{\sum_{i,j} a_{i,j}C_{i,j}} + rMN_1$.

\item[(F2)] Assume that $E \in \overline{\Pslicing}_{\beta, f^\ast\eta}((0, 1))$ has $H^{-1}(E)=0$ and $\ch_0(E)=0$. Write $\ch_1(E)=\Gamma + \sum_{i,j}b_{i,j}C_{i,j}$, where $\Gamma$ is effective and without components in $\Exc(f)$. Then $\norm{\sum_{i,j}b_{i,j}C_{i,j}} \leq L M \sqrt{\frac{A}{\eta^2}} (f^\ast\eta.\ch_1(E))$. 

\item[(F5)] Assume that $E \in \overline{\Pslicing}_{\beta, f^\ast\eta}((0, 1))$ has $H^{-1}(E)$ torsion-free of rank $r>0$, and that $\Supp H^0(E) \subseteq \Exc(f)$. Write $\ch_1(H^0(E)) = \sum_{i,j} b_{i,j}C_{i,j}$ and $\ch_1^\beta(H^{-1}(E))=f^\ast D + \sum_{i,j}a_{i,j} C_{i,j}$. We have that $\norm{\sum_{i,j}b_{i,j}C_{i,j}} \leq PM\norm{\sum_{i,j} a_{i,j}C_{i,j}} + rMN_5$.
\end{enumerate}
\end{cor}

\begin{proof}
We will only prove (F1); the other two follow in a similar fashion. To do so, let us apply the bound of Lemma \ref{lemma:bounds_ch1} for $j=1$. We get
\begin{align*}
b_{i,r_i} + r(k_{i,1}+\dots+k_{i,r_1}) &\leq \ch_1(E_{tf}) (C_{i,1}+\dots+C_{i,r_i}), \\
b_{i,j'} + r(k_{i,1}+\dots+k_{i,j'}) &\leq b_{i,j'+1} + \ch_1(E_{tf}) . (C_{i,1}+\dots+C_{i,j'}), 
\end{align*}
for $j' =1, \dots, r_i-1$. Adding these equations together with \eqref{eq:bounds_presmallkij} yields
\[ b_{i, j} \leq \ch_1^\beta(E_{tf})\sum_{a=j}^{r_i}(C_{i,1}+\dots+C_{i,a}) + \frac{r}{2} \sum_{a=j}^{r_i} n_{i,a}. \]
This gives us the bound
\[ \sum_{i,j} \abs{b_{i,j}} \leq \left( \sum_{i,j}a_{i,j} C_{i,j} \right)\left( \sum_{i,j}\sum_{a=j}^{r_i}(C_{i,1}+\dots+C_{i,a}) \right) + rN_1. \]
At last, the first term is bounded by $PM\sum_{i,j}\abs{a_{i,j}}$ using Cauchy--Schwarz.
\end{proof}

\begin{remark}
Let us point out that our argument is recursive, bounding $b_{i,j}$ in terms of $b_{i, j+1}$. A symmetrical argument can be perform to bound $b_{i, j}$ with $b_{i, j-1}$, which can be used to slightly improve the bounds on Corollary \ref{cor:bounds_F125}. These improvements are inconsequential for proving the support property. 

One could ask if we can remove the assumption on $\beta$ imposed in Corollary \ref{cond:bounds_smallkij}. For instance, we could try to prove a version of Lemma \ref{lemma:bounds_ch1} that does not involve Lemma \ref{lemma:bounds_objectsinTF}, at least not the ``strong'' result that $\O_{C_{i,j,j'}}(k_j, \dots, k_{j'}) \in \Ttors$. It is true that, e.g., $\O_{C_{i,j,j'}}(k_j, k_{j+1}+1, \dots, k_{j'}+1) \in \Ttors$ independently of the bound in Condition \ref{cond:bounds_smallkij}. However, this is simply not enough to get the extra terms while proving Corollary \ref{cor:bounds_F125}. 

Small improvements can be done, e.g., when each $r_i$ is at most two, to prove a version of Lemma \ref{lemma:bounds_ch1} without invoking Condition \ref{cond:bounds_smallkij}. We have omitted them from our exposition for the sake of clarity.
\end{remark}

\subsection{Putting all together} \label{subsec:wrapping}

We are finally ready to prove the support property for the pre-stability conditions $\overline{\sigma}_{\beta, f^\ast\eta}$. We start by considering the quadratic form
\begin{equation}
\begin{aligned}
Q_0(E) &= (f_\ast\ch_1^\beta(E))^2 + \epsilon(\ch_1^\beta(E)-f^\ast f_\ast \ch_1^\beta(E))^2 - 2\ch_0(E)\ch_2^\beta(E) \\
&+ \frac{B}{\eta^2}(f^\ast\eta.\ch_1^\beta(E))^2
\end{aligned}
\end{equation}
for some $\epsilon>0$ and $B>0$. Note that this is negative definite on the kernel of $Z_{\beta, sf^\ast\eta}$ for all $s>0$, by an application of the Hodge index theorem.

We now pick $\epsilon$ and $B$ as follows. We recall the constants $M, A, N_1, N_5, P, L, \delta$ from Subsection \ref{subsec:bounds}. Set $N= \max \{3, N_1, N_5, (\delta P)^5, (M/2\delta)^5\}$, $\epsilon = 1/N^3$, and $B=A(1+\epsilon L^2M^2)$. Let us point out that $0<\epsilon<1/3$, by the choice of $N$.

\begin{claim} \label{claim:bounds_choiceofcts}
The quadratic form $\left(1- \epsilon \right)x^2-\left( \frac{2\delta P}{N} +2\epsilon \right) xy + \left( \frac{2\delta}{MN}-\epsilon \right)y^2$ on the variables $x, y$, is positive definite. 
\end{claim}

Let us point out that these constants are computed explicitly just to highlight the fact that they depend only on $f\colon S \to T$ and the constants previously chosen---not on the object $E$ we will pick below.

\begin{lemma} \label{lemma:wrapping_alwaysss}
Let $E \in \overline{\Pslicing}_{\beta, f^\ast\eta}((0, 1])$ be an object that is $\sigma_s$-semistable for all $s \gg 0$. Assume that $E$ has phase $\neq 1$ for some (hence for all) $s$. Then $Q_0(E)\geq 0$. 
\end{lemma}

To prove this we will use the classification of Lemma \ref{lemma:bridgie_main}. Recall here that only the objects of (F1), (F2) and (F5) have slope smaller than one. We will prove this for the objects in (F1); the other two cases are handled similarly.

\begin{proof}[Proof of Lemma \ref*{lemma:wrapping_alwaysss}, (F1)]
Let us consider an object $E \in \overline{\Pslicing}_{\beta, f^\ast\eta}((0,1])$ as in Lemma \ref{lemma:bridgie_main} (F1). This way, we have that $H^{-1}(E)=0$, that $E=H^0(E)$ that is $f^\ast\eta$-semistable, whose torsion subsheaf lies in $\Ftors$. Let $E_{tors} \subseteq E$ be the torsion subsheaf of $E$, and let $E_{tf}=E/E_{tors}$, fitting into the short exact sequence $0 \to E_{tors} \to E \to E_{tf} \to 0$. Let us write $\ch_1^\beta(E_{tf}) = f^\ast D + \sum_{i,j} a_{i,j}C_{i,j}$ and $\ch_1(E_{tors}) = \sum_{i,j} b_{i,j}C_{i,j}$. Write also $r=\ch_0(E)=\ch_0(E_{tf})$. 

To start, the Bogomolov--Gieseker inequality applied to $E_{tf}$ implies
\begin{equation} \label{eq:wrapping_F1BG}
(f_\ast\ch_1^\beta(E_{tf}))^2 + (\ch_1^\beta(E_{tf})-f^\ast f_\ast \ch_1^\beta(E_{tf}))^2 - 2r \ch_2^\beta(E_{tf}) \geq 0.
\end{equation}
Here $f_\ast \ch_1^\beta(E_{tf})=f_\ast \ch_1^\beta(E)$, as $f_\ast \ch_1^\beta(E_{tors})=0$. Using this together with \eqref{eq:wrapping_F1BG} yields
\begin{align*}
& (f_\ast\ch_1^\beta(E))^2 + \epsilon (\ch_1^\beta(E)-f^\ast f_\ast \ch_1^\beta(E))^2 \\
\geq& 2r\ch_2^\beta(E) -2r\ch_2^\beta(E_{tors}) -(1-\epsilon)(\ch_1^\beta(E_{tf})-f^\ast f_\ast \ch_1^\beta(E_{tf}))^2 \\
&+ \epsilon\ch_1^\beta(E_{tors})^2 +2 \epsilon (\ch_1^\beta(E_{tf})-f^\ast f_\ast \ch_1^\beta(E_{tf})) \ch_1^\beta(E_{tors}).
\end{align*}
The second term can be bounded by Corollary \ref{cor:bounds_onExc} (as $E_{tors} \in \Ftors$), while the other three can be written using \eqref{eq:bounds_norm} (and Cauchy--Schwarz for the last one). This yields
\begin{equation} \label{eq:bounds_pfF1-1}
\begin{aligned}
\geq& 2r\ch_2^\beta(E) + 2r\delta \norm{\sum_{i,j}b_{i,j}C_{i,j}} +(1-\epsilon)\norm{\sum_{i,j}a_{i,j}C_{i,j}}^2 - \epsilon \norm{\sum_{i,j}b_{i,j}C_{i,j}}^2 \\
&-2\epsilon \norm{\sum_{i,j}a_{i,j}C_{i,j}} \norm{\sum_{i,j}b_{i,j}C_{i,j}}.
\end{aligned}
\end{equation}
Let us focus on the second term. By using Corollary \ref{cor:bounds_F125} (together with $N_1 \leq N$) we have
\[ r \geq \frac{1}{MN} \norm{\sum_{i,j} b_{i,j}C_{i,j}} - \frac{P}{N} \norm{\sum_{i,j}a_{i,j}C_{i,j}}. \]
Replacing this in \eqref{eq:bounds_pfF1-1}, together with Claim \ref{claim:bounds_choiceofcts}, gives $Q_0(E) \geq \frac{B}{\eta^2}(f^\ast\eta.\ch_1^\beta(E))^2$, which is clearly non-negative.
\end{proof}

\begin{cor} \label{cor:wrapping_supportnot1}
Let $E \in \overline{\Pslicing}_{\beta, f^\ast\eta}((0, 1])$ be an object that is $\overline{\sigma}_{\beta, f^\ast\eta}$-semistable of phase $\neq 1$. Then $Q_0(E)\geq 0$. 
\end{cor}

\begin{proof}
We proceed as in \cite{Tod13}*{3.23} (cf. \cite{MS17}*{6.13}). Note that $Q_0$ is negative definite on $\ker \sigma_s$ for all $s$. This way, we proceed by induction on $\Im Z_{\beta, f^\ast\eta}(E)>0$. (Note that $\Im Z_{\beta, f^\ast\eta}$ takes discrete values, as $\beta, \eta$ have rational coefficients.) The base case is Lemma \ref{lemma:wrapping_alwaysss}, and the induction step follows by \cite{MS17}*{5.29}.
\end{proof}

We are ready to prove the support property.

\begin{proof}[Proof of Theorem \ref{teo:intro_existence}, support property]
We claim that the pre-stability condition $\overline{\sigma}_{\beta, f^\ast\eta}$ satisfies the support property with respect to the quadratic form
\[ Q(E) = Q_0(E) + \frac{\epsilon}{\delta^2}\Re Z_{\beta, f^\ast\eta}(E)^2. \]
Note that $Q$ agrees with $Q_0$ on $\ker Z_{\beta, f^\ast\eta}$, hence $Q$ is negative definite on it. Moreover, we have shown in Corollary \ref{cor:wrapping_supportnot1} that $Q_0(E) \geq 0$ if $E$ is $\overline{\sigma}_{\beta, f^\ast\eta}$-semistable of phase not equal to 1. It remains to see what happens for $\overline{\sigma}_{\beta, f^\ast\eta}$-semistable objects of phase $1$. 

We can reduce to the case where $E$ is stable of phase 1. In this case either $E=H^0(E)$ or $E=H^{-1}(E)[1]$. In the first base we have $E$ is torsion by Lemma \ref{lemma:hslope_torsIm0}. By taking the zero-dimensional subsheaf, we reduce it to the case where $E=\O_x$ or $E \in \Ttors$ is pure of dimension one. The first case has $Q(E) = \frac{\epsilon}{\delta^2}>0$. The second one follows immediately by Corollary \ref{cor:bounds_onExc}.

At last, if $E = H^{-1}(E)[1]$, we can apply Corollary \ref{cor:hslope_HNfactors} to assume $E$ is either torsion or torsion-free. In the first case it follows by Corollary \ref{cor:bounds_onExc}, while the second one is direct from the Bogomolov--Gieseker inequality.
\end{proof}

\bibliography{Half1.bbl}
\end{document}